\documentclass[11pt,a4paper,reqno]{amsart}

\usepackage[margin=3cm]{geometry}
\usepackage{amsmath,amsfonts,amssymb,amsthm}
\usepackage{mathtools}
\usepackage{url}
\usepackage[usenames,dvipsnames]{xcolor}
\usepackage[colorlinks=true,linkcolor=Blue,citecolor=Green,unicode]{hyperref}
\usepackage{aliascnt}       
\usepackage[T1]{fontenc}
\usepackage[utf8]{inputenc}
\usepackage{enumitem}
\setlist[enumerate]{leftmargin=*}
\usepackage{textgreek}
\usepackage[capitalise]{cleveref}
\usepackage{algpseudocode}
\usepackage{algorithm}
\usepackage{caption}
\usepackage[textwidth=2.5cm,textsize=small]{todonotes}

\newtheorem{theorem}{Theorem}[section]
\newaliascnt{lemma}{theorem}
\newaliascnt{corollary}{theorem}
\newaliascnt{definition}{theorem}
\newaliascnt{remark}{theorem}
\newaliascnt{proposition}{theorem}
\newaliascnt{conjecture}{theorem}
\newaliascnt{example}{theorem}
\newaliascnt{question}{theorem}
\newaliascnt{claim}{theorem}

\newtheorem{lemma}[lemma]{Lemma}
\aliascntresetthe{lemma}

\newtheorem*{lemma*}{Lemma}

\aliascntresetthe{corollary}

\newtheorem*{corollary*}{Corollary}

\newtheorem{definition}[definition]{Definition}
\aliascntresetthe{definition}

\newtheorem*{definition*}{Definition}

\newtheorem{remark}[remark]{Remark}
\aliascntresetthe{remark}

\newtheorem*{remark*}{Remark}

\newtheorem{proposition}[proposition]{Proposition}
\aliascntresetthe{proposition}

\newtheorem*{proposition*}{Proposition}

\newtheorem{conjecture}[conjecture]{Conjecture}
\aliascntresetthe{conjecture}

\newtheorem*{conjecture*}{Conjecture}

\aliascntresetthe{example}

\newtheorem*{example*}{Example}

\newtheorem*{problem*}{Problem}

\newtheorem{question}[question]{Question}
\aliascntresetthe{question}

\newtheorem*{question*}{Question}

\aliascntresetthe{claim}

\newtheorem*{claim*}{Claim}

\Crefname{figure}{Figure}{Figures}

\DeclareMathOperator{\conv}{conv}

\DeclareMathOperator{\diag}{diag}

\def\cG{\mathcal{G}}
\def\cH{\mathcal{H}}

\def\cO{\mathcal{O}}

\def\R{\mathbb{R}}

\def\Z{\mathbb{Z}}

\def\N{\mathbb{N}}
\def\Q{\mathbb{Q}}

\DeclareMathOperator{\vol}{vol}
\DeclareMathOperator{\Vol}{Vol}

\DeclarePairedDelimiter{\card}{\lvert}{\rvert}
\newcommand{\abs}[1]{\left|#1\right|}

\newcommand{\zero}{\mathbf{0}}
\newcommand{\dif}{\mathop{}\!\mathrm{d}}

\DeclareMathOperator{\rank}{rk}
\DeclareMathOperator{\fct}{fct}
\DeclareMathOperator{\adj}{adj}
\DeclareMathOperator{\h}{h}
\DeclareMathOperator{\g}{g}
\DeclareMathOperator{\s}{s}
\DeclareMathOperator{\gv}{gv}
\DeclareMathOperator{\w}{w}
\DeclareMathOperator{\wv}{wv}

\newcommand{\floor}[1]{\left\lfloor #1\right\rfloor}
\newcommand{\ceil}[1]{\left\lceil #1 \right\rceil}

\newcommand{\CC}{{C\nolinebreak[4]\hspace{-.05em}\raisebox{.4ex}{\tiny\textbf{++}}}}

\numberwithin{equation}{section}

\begin{document}

\title[Size of integer programs with bounded non-vanishing subdeterminants]{On the size of integer programs with bounded non-vanishing subdeterminants}

\author{Bj\"orn Kriepke}
\address{Institute of Mathematics\\
 University of Rostock\\
 Germany}
\email{bjoern.kriepke@uni-rostock.de}

\author{Gohar M.~Kyureghyan}
\address{Institute of Mathematics\\
 University of Rostock\\
 Germany}
\email{gohar.kyureghyan@uni-rostock.de}

\author{Matthias Schymura}
\address{Institute of Mathematics\\
 University of Rostock\\
 Germany}
\email{matthiasmschymura@gmail.com}





\begin{abstract}
Motivated by complexity questions in integer programming, this paper aims to contribute to the understanding of combinatorial properties of integer matrices of row rank~$r$ and with bounded subdeterminants.
In particular, we study the column number question for integer matrices whose every $r \times r$ minor is non-zero and bounded by a fixed constant~$\Delta$ in absolute value.
Approaching the problem in two different ways, one that uses results from coding theory, and the other from the geometry of numbers, we obtain linear and asymptotically sublinear upper bounds on the maximal number of columns of such matrices, respectively.
We complement these results by lower bound constructions, matching the linear upper bound for $r=2$, and a discussion of a computational approach to determine the maximal number of columns for small parameters $\Delta$ and $r$.
\end{abstract}

\maketitle

\section{Introduction}

This paper contributes to the understanding of combinatorial properties of integer programs, parametrized by the maximal absolute value of a subdeterminant of their constraint matrices.
An \emph{integer program} is an optimization problem of the form
\begin{align}
\max\left\{ c^\intercal x : Bx \leq d, x \in \Z^r \right\},\tag{IP}
\end{align}
where $B \in \Z^{n \times r}$ is the constraint matrix, $c \in \Z^r$ the cost function, and $d \in \Z^n$ the right hand side.
The underlying polyhedron of this integer program is defined as
\[
P(B,d) := \left\{ x \in \R^r : Bx \leq d \right\},
\]
that is, as the set of all solutions of the linear inequality system $b_1^\intercal x \leq d_1, \ldots, b_n^\intercal x \leq d_n$ with $b_1,\ldots,b_n \in \Z^r$ denoting the rows of~$B$, which we compactly write as $Bx \leq d$.

Solving a general integer program is computationally expensive.
Indeed the associated decision problem that asks whether $P(B,d) \cap \Z^r \neq \emptyset$ is among Karp's~\cite{karp1972reducibility} famous early list of NP-complete problems.
The most efficient algorithm for solving a general instance of the integer program~(IP) was very recently proposed in~\cite{reisrothvoss2023thesubspace}, and has a running time of $\log(2r)^{\cO(r)}$ multiplied with a polynomial in the encoding length of~$B,d$, and~$c$.
We refer to that paper~\cite{reisrothvoss2023thesubspace} for a detailed discussion of the history of such complexity results.
The question whether~(IP) can be solved in single exponential time is one of the major research problems in the area.

Given the computational hardness of integer programming, there is an extensive amount of research about investigating integer programs with additional structure.
The most basic assumption that implies efficient algorithms is that the constraint matrix~$B$ in~(IP) is \emph{totally unimodular}, which means that all its \emph{minors}, that is, the determinants of square submatrices, are in $\{-1,0,1\}$.
As a consequence, the polyhedron $P(B,d)$ has only integral vertices and thus~(IP) can be solved in polynomial time, for instance, by linear programming techniques (see~\cite[Ch.~19]{schrijver1986theory}).

This observation suggests to parametrize integer programs by maximal determinants of the submatrices of the constraint matrix.
More precisely, given an integer $\Delta \in \Z_{>0}$, we say that a \emph{rational} matrix $A \in \Q^{r \times n}$ of rank $\rank(A)$ is \emph{$\Delta$-submodular}, if all its $\rank(A) \times \rank(A)$ minors are bounded by $\Delta$ in absolute value.
Further, we say that~$A$ is \emph{$\Delta$-modular}, if it is $\Delta$-submodular and there is a $\rank(A) \times \rank(A)$ minor equal to~$\Delta$ or~$-\Delta$.
If the minors of \emph{every} size are bounded by~$\Delta$ in absolute value, we call~$A$ \emph{totally $\Delta$-(sub)modular}.
An integer program~(IP) with a $\Delta$-modular constraint matrix~$B$ is called a \emph{$\Delta$-modular integer program}.
Note that this definition is independent of the right hand side~$d \in \Z^n$.

As discussed above, totally unimodular ($\Delta=1$) integer programs are efficiently solvable.
This is relevant for problems that can be formulated by network matrices, for instance (see~\cite[Ch.~19]{schrijver1986theory}), but at the same time it is a very restrictive assumption.
Extending our understanding of $\Delta$-modular integer programs beyond the totally unimodular case is a prevailing and currently very active research direction in the community.
The main conjecture in the field is that efficient algorithms exist, whenever~$\Delta$ is not part of the input:

\begin{conjecture}
\label{conj:ip-subdeterminants}
Let $\Delta \in \Z_{>0}$ be fixed.
Then, there is a polynomial time algorithm to solve any integer program of the form~(IP) with a $\Delta$-modular constraint matrix.
\end{conjecture}

The origin of this conjecture is hard to trace, but~\cite{gribanovmalyshevpardalosveselov} attribute it to~\cite{shevchenko1997qualitative}; it might be much older than that though.
Despite the fact that \cref{conj:ip-subdeterminants} is far from being resolved in whole generality, there is a list of pertinent results supporting it and which cover different variations and specializations of the problem.
In fact, polynomial time algorithms have been devised for $2$-modular integer programs~\cite{artmannweismantelzenklusen2017astrongly}, for totally $\Delta$-modular integer programs with at most two non-zeros per row~\cite{fiorinijoretweltgeyuditsky2021integer}, and for $\Delta \leq 4$ for so-called strictly $\Delta$-modular integer programs~\cite{naegelesantiagozenklusen2021congruency}.

Additionally, the class of \emph{generic} $\Delta$-modular integer programs is well understood from a computational point of view.
A matrix $A \in \Z^{r \times n}$ is called \emph{generic}, if all its $\rank(A) \times \rank(A)$ minors are non-zero, that is, every set of~$\rank(A)$ columns or rows of~$A$ is linearly independent.
Generic $\Delta$-modular integer programming, for arbitrary but fixed $\Delta \in \Z_{>0}$, is known to be polynomial time solvable due to~\cite{artmanneisenbrandglanzeroertelvempalaweismantel2016anote} (the case of generic $2$-modular constraint matrices was known before~\cite{veselovchirkov2009integer}).
More strongly, all the potentially relevant feasible points in the underlying polyhedron of a generic $\Delta$-modular integer program can be enumerated efficiently:

\begin{theorem}[{\cite[Thm.~2.1]{jiangbasu2022enumerating}}]
Let $\Delta \in \Z_{>0}$ be fixed.
Then, for every generic $\Delta$-modular matrix $B \in \Z^{n \times r}$ and every $d \in \Z^n$, one can enumerate the vertices of the convex hull of $P(B,d) \cap \Z^r$ in polynomial time.
\end{theorem}

In this paper, we continue the study of generic $\Delta$-modular integer programs and their combinatorial properties.
Our aim is to underpin the algorithmic results for this class of optimization problems by corresponding strong combinatorial obstructions.
More precisely, we are interested in the maximum number of irredundant inequality constraints that a generic $\Delta$-modular integer program can have.
After interchanging rows and columns (which is more convenient for our analysis) this corresponds to bounding the number of columns of a generic $\Delta$-modular integer matrix for a fixed number~$r$ of rows.
We may also assume that the considered matrices have \emph{no repeating columns}, and we will do so for the rest of the paper without further mention.
By this assumption we can identify a matrix $A \in \Z^{r \times n}$ with its set of columns and then consider it as a subset of $\Z^r$ of cardinality~$n$.

The combinatorial question described above is the so-called \emph{column number question} for generic $\Delta$-modular integer matrices and is made precise by the function
\begin{align*}
\g(\Delta,r) &:= \max\left\{ n : \exists \textrm{ a generic }\Delta\textrm{-modular }A \in \Z^{r \times n}\textrm{ with} \rank(A) = r \right\}.
\end{align*}
The relation to the combinatorial complexity of the corresponding~(IP)s is given by
\begin{align}
\max\left\{ \fct(P(B,d)) : B^\intercal \in \Z^{r \times n} \textrm{ is generic } \Delta\textrm{-modular}, d \in \Z^n \right\} = 2 \cdot \g(\Delta,r),\label{eqn:g-facet-number}
\end{align}
where $\fct(P)$ denotes the number of facets of the polyhedron~$P$.
This identity can be seen by taking suitably large entries for the vector~$d$ of right hand sides, so that every given constraint corresponds to a facet-defining inequality of the polyhedron $P(B,d)$.
The factor of two arises as the definition of $\g(\Delta,r)$ excludes parallel columns, but the polyhedron $P(B,d)$ can have pairs of facets that are supported by parallel hyperplanes with opposite outer normal vectors.

Our main results on the behavior of the function~$\g(\Delta,r)$ substantiate the fact that generic $\Delta$-modular integer programs are quite restrictive.
They are summarized as follows:

\begin{theorem}
\label{thm:linear-upper-bound}
Let $r$ and $\Delta$ be positive integers with $r \geq 2$, and let $p$ be the smallest prime number with $\Delta < p$. Then,
\begin{align}
\g(\Delta,r) \leq \max\{r, p\} + 1.\label{eqn:linear-bound}
\end{align}
Furthermore, if $r,\Delta \geq 2$, then $\g(\Delta,r) \leq 2\Delta$, and if $r \geq 2\Delta - 1$, then $\g(\Delta,r) = r+1$.
\end{theorem}

For $r=2$, the bound~\eqref{eqn:linear-bound} appears in \cite{paatstallknechtwalshxu2022onthecolumn}, for $r \geq 3$ it is apparently new.
The fact that $\g(\Delta,r) = r+1$, if $r$ is large enough compared to~$\Delta$, was established before in~\cite[Lem.~7]{artmanneisenbrandglanzeroertelvempalaweismantel2016anote}, albeit with the much stronger assumption that~$r$ is at least doubly exponential in~$\Delta$.
A direct consequence of~\eqref{eqn:linear-bound} is that $\g(\Delta,r)$ is at most linear in \emph{both} parameters~$\Delta$ and~$r$.
The previous state of the art was given by the inequality $\g(\Delta,r) \leq r + \Delta^2$, established in~\cite[Lem.~8]{paatschloeterweismantel2022theintegrality} (cf.~\cite[Proof of Lem.~1]{glanzerstallknechtweismantel2022notesonabc} for a similar bound).
The authors of~\cite{paatschloeterweismantel2022theintegrality} show how bounds on $\g(\Delta,r)$ have consequences on what they term the \emph{integrality number} of the corresponding integer programs.

As our second main result we show that for fixed rank~$r \geq 3$ the bound~\eqref{eqn:linear-bound} is not tight and that the dependence on~$\Delta$ is indeed \emph{sublinear}.

\begin{theorem}
\label{thm:sublinear-bound-generic-heller}
For every $r,\Delta \in \Z_{>0}$ with $r \geq 3$, we have
\[
\g(\Delta,r) \leq 130\, r^3 \Delta^{\frac{2}{r}}.
\]
\end{theorem}

\noindent The upper bounds in \cref{thm:linear-upper-bound,thm:sublinear-bound-generic-heller} are proven in \cref{sect:upper-bounds} using tools from coding theory and the geometry of numbers, respectively.

We complement these results by lower bound constructions that are best possible for $r=2$, and that imply that the exponent $\frac{2}{r}$ in \cref{thm:sublinear-bound-generic-heller} is correct up to a constant factor.
The essence of these constructions, whose details are layed out in \cref{sect:constr-two-rows} and \cref{prop:generic-gen-heller-lower-bound}, is subsumed into the following (slightly informal) statement:

\begin{theorem}
\label{thm:lower-bounds-subsumed}\
\begin{enumerate}[label=(\roman*)]
 \item For $r = 2$, there are at least three infinite families of integers~$\Delta$ for which the bound~\eqref{eqn:linear-bound} is attained.
 \item For fixed $r \geq 3$, the function $\Delta \mapsto \g(\Delta,r)$ grows at least with the order $\Omega(\Delta^\frac{1}{r-1})$.
\end{enumerate}
\end{theorem}

\noindent We suspect that these three families are the only infinite families meeting the bound~\eqref{eqn:linear-bound}.
\noindent Additionally to these theoretical results on the asymptotic behavior of the counting function $(\Delta,r) \mapsto \g(\Delta,r)$, we devised an algorithm to compute $\g(\Delta,r)$ for small parameters~$r$ and~$\Delta$.
The algorithm, its specifications and the computational results obtained will be described in detail in \cref{sect:computational-approach}.

\medskip
Of course, the column number question has been investigated also for general, not necessarily generic, $\Delta$-modular integer matrices.
The corresponding counting function is
\[
\h(\Delta,r) := \max\left\{ n : \exists \textrm{ a }\Delta\textrm{-modular }A \in \Z^{r \times n}\textrm{ with} \rank(A) = r \right\}.
\]
Note that in contrast to $\g(\Delta,r)$ this allows linear dependencies among sets of~$r$ columns in the considered matrices, and in particular, it allows for parallel columns.
We shortly mention the best-known bounds on this parameter and refer to the cited papers for more background and literature.

A construction of a $\Delta$-modular matrix with many columns in~\cite{leepaatstallknechtxu2022polynomial} shows that
\begin{align}
r^2 + r + 1 + 2r(\Delta - 1) \leq \h(\Delta,r), \textrm{ for all } \Delta,r \in \Z_{>0}.\label{lee-et-at-general-bound}
\end{align}
Furthermore, the authors of~~\cite{leepaatstallknechtxu2022polynomial} prove that $\h(\Delta,r) \leq (r^2+r)\Delta^2 + 1$ and that the lower bound in~\eqref{lee-et-at-general-bound} is tight whenever $\Delta \leq 2$ or $r \leq 2$.
They also conjectured that the identity $\h(\Delta,r) = r^2 + r + 1 + 2r(\Delta - 1)$ holds for \emph{every}~$r$ and~$\Delta$, which has been disproven in~\cite{averkovschymura2023onthemaximal} for $\Delta \in \{4,8,16\}$ and~$r$ large enough.
However, the quantitative version of this conjecture, that is, whether $\h(\Delta,r) \leq r^2 + \cO(r\cdot\Delta)$, still stands.
A step in this direction was achieved in~\cite{averkovschymura2023onthemaximal} (see also~\cite{averkovschymura2022maximal-ipco} for a less technical extended abstract), in which a bound of order $\h(\Delta,r) \in \cO(r^4) \cdot \Delta$, for $r \geq 5$ and $\Delta \in \Z_{>0}$ was obtained.
Together with~\eqref{lee-et-at-general-bound} this shows that, for fixed~$r$, the function $\Delta \mapsto \h(\Delta,r)$ grows linearly.
%

Although the function $\h(\Delta,r)$ is apparently closely connected to the maximal possible number of irredundant inequalities in a general $\Delta$-modular integer program, allowing for parallel columns in the definition of a $\Delta$-modular integer matrix prohibits the exact relationship analogous to~\eqref{eqn:g-facet-number}.
This can be circumvented by considering what we call \emph{simple}\footnote{We borrow this term from matroid theory, as the matroid represented by a simple matrix is a simple matroid.} matrices, in which only non-zero and pairwise non-parallel columns are allowed.
We denote the counting function by
\[
\s(\Delta,r) := \max\left\{ n : \exists \textrm{ a simple }\Delta\textrm{-modular }A \in \Z^{r \times n}\textrm{ with} \rank(A) = r \right\},
\]
and observe that now the exact relationship
\begin{align}
\max\left\{ \fct(P(B,d)) : B^\intercal \in \Z^{r \times n} \textrm{ is } \Delta\textrm{-modular}, d \in \Z^n \right\} = 2 \cdot \s(\Delta,r).\label{eqn:s-facet-number}
\end{align}
holds, analogous to~\eqref{eqn:g-facet-number} in the generic case.
Since every generic matrix is necessarily simple and simplicity restricts the general case, we have
\[
\g(\Delta,r) \leq \s(\Delta,r) \leq \h(\Delta,r) \qquad \textrm{and} \qquad \s(\Delta,2) = \g(\Delta,2).
\]
In particular, the upper bounds on~$\h(\Delta,r)$ discussed previously also apply to~$\s(\Delta,r)$.
Additionally, in~\cite{paatstallknechtwalshxu2022onthecolumn} the inequality
\[
\s(\Delta,r) \leq \tbinom{r+1}{2} + 80 \Delta^7 \cdot r, \textrm{ for every } \Delta \in \Z_{>0} \textrm{ and for } r \textrm{ sufficiently large},
\]
is derived.
Their arguments are based on matroid theory and are not applicable to~$\h(\Delta,r)$.
%

In conclusion, comparing the generic case, in particular \cref{thm:sublinear-bound-generic-heller}, to the general or simple setting, we see that $\g(\Delta,r)$ is separated from $\s(\Delta,r)$ and $\h(\Delta,r)$ in their asymptotic behavior, as expected.

\subsection*{Notations}

Here, we fix some notations that we use throughout the following sections.
For a positive integer~$r$, we write $[r] = \{1,2,\ldots,r\}$.
If $M \subseteq \R^r$ is a Lebesgue-measurable set, then we denote its Lebesgue measure, or \emph{volume}, by $\vol(M)$, and its \emph{normalized volume} by $\Vol(M) = r! \vol(M)$.
The symbol $\conv(S)$ stands for the \emph{convex hull} of a set $S \subseteq \R^r$, and if~$S$ is finite then $\card{S}$ denotes its cardinality.
For a matrix~$A$, we let $\rank(A)$ denote its rank, as used already above.
Finally, an integer matrix $U \in \Z^{r \times r}$ with determinant equal to~$1$ or~$-1$ is called a \emph{unimodular matrix}.

\section{Upper bounds on \texorpdfstring{$\g(\Delta,r)$}{g(\textDelta,r)}}
\label{sect:upper-bounds}

\subsection{A linear upper bound using finite fields}
\label{sect:mds}

A generic $\Delta$-modular matrix implying the lower bound $\g(\Delta,r) \geq r+1$ has columns $e_1,\ldots,e_r,(\Delta,\ldots,\Delta)^\intercal$.
In this part we aim to prove \cref{thm:linear-upper-bound}, which implies that this lower bound is best possible, if $r \geq 2\Delta - 1$.

Our proof of \cref{thm:linear-upper-bound} relies on a connection of generic $\Delta$-modularity with so-called \emph{MDS-codes}.
The famous MDS conjecture describes possible parameters of an optimal family of codes over the finite field with $q$ elements, denoted here by $\mathbb{F}_q$. It was first stated by Segre in the $1950$s in terms of arcs in finite geometry. 
For $q=p$ prime the MDS conjecture was proven by Ball~\cite{ball2012sets}: 

\begin{theorem}[{Ball~\cite[Cor.~11.1]{ball2012sets}}]
\label{thm:ball-mds}
Let $p$ be a prime and let $A\in\mathbb{F}_p^{r\times n}$ be a matrix with $\rank(A)=r\geq 2$ such that any set of $r$ columns of~$A$ is linearly independent.
Then $n\leq \max\{r, p\}+1$.
\end{theorem}

Note that the weaker bound $n\leq r + p - 1$, for arbitrary~$p$ and~$r$ as above, and the bound $n \leq r+1$, for $p \leq r$, are much easier to obtain (cf.~\cite[Lem.~1.2 \& Lem.~1.3]{ball2012sets}).
The case $p > r$ is where the difficulty in \cref{thm:ball-mds} lies.
We refer the interested reader to the book~\cite{ball2015finite} for an excellent reference on MDS codes and their connections to finite geometry.

\begin{proof}[Proof of \cref{thm:linear-upper-bound}]
Let $A\in\Z^{r\times n}$ be a generic $\Delta$-modular matrix with $\rank(A) = r$.
We interpret $A$ as a matrix over $\mathbb F_p$.
Let $A_I$ be any $r\times r$ submatrix of~$A$, then 
\begin{equation*}
  0 < \abs{\det(A_I)} \leq \Delta < p
\end{equation*}
as integers and thus $\det(A_I)$ is nonzero in~$\mathbb F_p$.
Therefore, any set of $r$ columns of $A$ is linearly independent over~$\mathbb F_p$ and by applying \cref{thm:ball-mds} we obtain $n\leq \max\{r, p\}+1$, which implies the claim.

Using Bertrand's postulate on the existence of a prime~$p$ with $\Delta < p < 2\Delta$, for $\Delta \geq 2$, shows that $\g(\Delta,r) \leq 2 \Delta$, for every $r \geq 2$, and $\g(\Delta,r) = r+1$, if $r \geq 2\Delta - 1$.
\end{proof}

In the coding theory setting, generator matrices of so-called \emph{Reed-Solomon codes} meet the upper bound of \cref{thm:ball-mds} (cf.~\cite[Cor.~9.2]{ball2012sets}).
However, in general, Reed-Solomon codes do not translate into generic $\Delta$-modular integer matrices.
In fact, in the integer setting the bound of \cref{thm:linear-upper-bound} is not always tight, which we demonstrate numerically for $r=2$ in \cref{sect:numerical-results}, and prove for $r\geq 3$ in the following subsection.

Nonetheless we use Reed-Solomon codes implicitly to derive a construction of order $\Omega(\Delta^{\frac{1}{r-1}})$ in \cref{sect:vandermonde-construction}. This construction is stated in terms of Vandermonde-matrices which, in fact, generate Reed-Solomon codes.

\subsection{A sublinear upper bound for fixed \texorpdfstring{$r \geq 3$}{r >= 3} using geometry of numbers}

In this part, we prove \cref{thm:sublinear-bound-generic-heller}.
The main idea of our argument is the following:
Given a generic $\Delta$-modular matrix $A \in \Z^{r \times n}$, considered as a subset of $\Z^r$ with cardinality~$n$, we upper bound~$n$ by counting the points in~$A$ by containment in a family of parallel affine hyperplanes.
The subsets of~$A$ in each of these hyperplanes are points in general position and that satisfy a volume condition in terms of the parameter~$\Delta$.
These point sets in $r-1$ dimensions again can be counted by containment in a family of parallel hyperplanes, only that now at most $r-1$ of the points can be contained in a given such hyperplane by the general position condition.
For the envisioned bound, we need to find a family of parallel hyperplanes of small cardinality and that covers the points in~$A$.
This is achieved by using the $\Delta$-modularity in form of a volume-width-principle from the geometry of numbers.

The sketched argument can be implemented in a series of lemmas.
First of all, let $S \subseteq \Z^r$ be a finite set of lattice points.
\begin{itemize}
 \item $S$ is in \emph{general position}, if no $r+1$ of its points are contained in a common affine hyperplane, and 
 \item $S$ has \emph{simplex-volume at most~$\Delta$}, if for every $r+1$ points in~$S$ the lattice simplex spanned by those points has normalized volume at most~$\Delta$.
\end{itemize}
With this notation we define the constant
\[
\gv(\Delta,r) := \max\left\{ \abs{S} : S \subseteq \Z^r \textrm{ in general position with simplex-volume at most } \Delta \right\}.
\]
A simple observation is the identity $\gv(\Delta,1) = \Delta + 1$, realized by any set of $\Delta+1$ consecutive integers.
We will estimate both $\gv(\Delta,r)$ and $\g(\Delta,r)$ by counting points by parallel lattice hyperplanes in~$\R^r$.
In order to do so, we need the concept of lattice-width.
For a non-zero vector $v \in \R^r \setminus \{\zero\}$ the \emph{width} of a convex body~$K \subseteq \R^r$ in direction~$v$ is defined as
\[
\omega(K,v) = \max_{x \in K} x^\intercal v - \min_{x \in K} x^\intercal v.
\]
Minimizing the width over all non-zero lattice directions yields the \emph{lattice-width}
\[
\omega_L(K) = \min_{v \in \Z^r \setminus \{\zero\}} \omega(K,v).
\]
If $K$ is a \emph{lattice polytope}, meaning that $K = \conv(S)$, for a finite set $S \subseteq \Z^r$, then for $v \in \Z^r$ with $\gcd(v_1,\ldots,v_r) = 1$, the width of~$K$ in direction~$v$ can be understood as follows:
$\omega(K,v)+1$ is the number of parallel lattice-planes orthogonal to~$v$ and which intersect~$K$ non-trivially.

For any finite set $S \subseteq \R^r$ we write $P_S := \conv(S)$ for the polytope defined as the convex hull of the points in~$S$.
Coming back to $\Delta$-modularity, we define
\[
\w(\Delta,r) = \max\left\{ \omega_L(P_A) : A \subseteq \Z^r \textrm{ generic }\Delta\textrm{-modular} \right\}
\]
as the maximal lattice-width of a lattice polytope that arises as the convex hull of the columns of a generic $\Delta$-modular matrix with~$r$ rows.
Likewise, we define
\[
\wv(\Delta,r) = \max\left\{ \omega_L(P_S) : S \subseteq \Z^r \textrm{ in general position with simplex-volume}\leq\Delta \right\}
\]
as the maximal lattice-width of a lattice polytope defined by a subset of~$\Z^r$ in general position and of simplex-volume at most~$\Delta$.

Our motivation to study these quantities is their intimate relationship to the quantities $\g(\Delta,r)$ and $\gv(\Delta,r)$, which is expressed by the following lemma.
For its statement, we need the notation $\abs{\ell}_+ = \max\{1,\abs{\ell}\}$, for an integer $\ell \in \Z$.

\goodbreak
\begin{lemma}
\label{lem:slice-reduction}
For every $r,\Delta \in \N$, we have
\begin{enumerate}[label=(\roman*)]
 \item $\g(\Delta,r) \leq \sup_{a \in \Z} \sum_{\ell = a}^{a+ \w(\Delta,r)} \gv\left(\frac{\Delta}{\abs{\ell}_+},r-1\right)$ and
 \item $\gv(\Delta,r) \leq (\wv(\Delta,r) + 1) \cdot r$.
\end{enumerate}
\end{lemma}

\begin{proof}
(i): Let $A \subseteq \Z^r$ be a generic $\Delta$-modular set of lattice points.
We seek to upper bound~$\abs{A}$.
Let $v \in \Z^r$ be a direction that attains the lattice-width of $P_A$, meaning that $\omega_L(P_A) = \omega(P_A,v)$.
We may apply a suitable unimodular transformation and assume that $v = e_r$.
Writing $L = e_r^\perp$ for the hyperplane orthogonal to~$e_r$, we find that each point of~$A$ is contained in one of the $\omega_L(P_A)+1$ parallel lattice planes
\[
L + a e_r, L + (a+1) e_r, \ldots, L + b e_r,
\]
where $a,b \in \Z$ are such that $a < b$ and $b-a = \omega_L(P_A)$.

For $\ell \in \{a,a+1,\ldots,b\}$ we count the points in $A \cap (L + \ell e_r)$.
If $\ell=0$, then $\abs{A \cap L} \leq r-1$, because $A$ is assumed to be generic, so no~$r$ points of~$A$ can be contained in the same linear subspace.
Also note that $r-1 \leq \gv(1,r-1) \leq \gv(\Delta,r-1)$.

So, we may assume that $\ell \neq 0$ and we let $C_\ell = A \cap (L + \ell e_r)$ be the set of points in~$A$ with last entry equal to~$\ell$.
Further, let $S_\ell \subseteq \Z^{r-1}$ be the projection of~$C_\ell$ that forgets the last coordinate.
The set~$S_\ell$ is in general position because~$A$ is generic.
Moreover, let $c_1,\ldots,c_r \in C_\ell$ be a choice of~$r$ lattice points in~$C_\ell$ with corresponding projections $s_1,\ldots,s_r \in S_\ell$.
Then,
\begin{align*}
\Delta \geq \abs{\det(c_1,\ldots,c_r)} &= \abs{\det\left(\begin{matrix}
s_1 & s_2 & \ldots & s_r \\ \ell & \ell & \ldots & \ell
\end{matrix}\right)}
= \abs{\ell} \cdot \abs{\det(s_2-s_1,\ldots,s_r-s_1)}\\
&= \abs{\ell} \cdot \Vol_{r-1}(\conv(\{s_1,\ldots,s_r\})) > 0.
\end{align*}
This implies that the set $S_\ell \subseteq \Z^{r-1}$ has simplex-volume at most $\frac{\Delta}{\abs{\ell}} \leq \Delta$ and thus
$\abs{C_\ell} = \abs{S_\ell} \leq \gv(\Delta/\abs{\ell},r-1)$.

In summary we obtain the desired inequality
\[
\abs{A} = \sum_{\ell = a}^b \abs{A \cap (L + \ell e_r)} \leq \sum_{\ell = a}^b \gv\left(\frac{\Delta}{\abs{\ell}_+},r-1\right) \leq \sup_{a \in \Z} \sum_{\ell = a}^{a+ \w(\Delta,r)} \gv\left(\frac{\Delta}{\abs{\ell}_+},r-1\right),
\]
where the notation $\abs{\ell}_+$ takes care of the case $\ell=0$.

(ii): The claimed upper bound on $\gv(\Delta,r)$ follows from a similar argument as the one in part (i).
Let $S \subseteq \Z^r$ be a point set in general position and with simplex-volume at most~$\Delta$.
Further, let $v \in \Z^r$ be a direction attaining the lattice-width of~$P_S$.
Since~$S$ is in general position, each hyperplane can contain at most~$r$ points of~$S$, and hence
\[
\abs{S} \leq r \cdot (\omega(P_S,v) + 1) = r \cdot (\omega_L(P_S) + 1) \leq r \cdot (\wv(\Delta,r) + 1),
\]
and we are done.
\end{proof}

Hence, in order to upper bound the function $\g(\Delta,r)$, we need to upper bound both $\w(\Delta,r)$ and $\wv(\Delta,r)$.
We do so by applying a volume argument which is based on the intuition that a convex body of bounded volume necessarily has small width in \emph{some} direction.
This intuition is made precise by the following two inequalities:
\begin{align}
\vol(K) &\geq \left(\frac{\pi}{8}\right)^r \frac{r+1}{2^r r!}\, \omega_L(K)^r,\quad \textrm{ for every convex body } K \subseteq \R^r,\label{eqn:makai-conj}
\end{align}
and
\begin{align}
\vol(K) &\geq \left(\frac{\pi}{4}\right)^r \frac{1}{r!}\, \omega_L(K)^r,\quad \textrm{ for every convex body } K \subseteq \R^r \textrm{ with } K=-K.\label{eqn:makai-conj-symmetric}
\end{align}
The condition $K=-K$ in~\eqref{eqn:makai-conj-symmetric} means that~$K$ is invariant under reflecting any of its points in the origin.
These asymptotic bounds can be derived from the asymptotics on Mahler's famous conjecture on the minimal volume-product of convex bodies with the symmetry in~\eqref{eqn:makai-conj-symmetric}.
Makai Jr.~\cite{makai1978on} conjectures that the factor $(\frac{\pi}{8})^r$ in~\eqref{eqn:makai-conj} and $(\frac{\pi}{4})^r$ in~\eqref{eqn:makai-conj-symmetric} can be replaced by one.
%
We refer the reader to~\cite[Thm.~II]{alvarezpaivabalachefftzanev2016isosystolic} and~\cite{gonzalezschymura2017ondensities} in which these bounds are described and more background is given.

With these tools at hand, we can now determine the asymptotics of the parameters $\w(\Delta,r)$ and $\wv(\Delta,r)$:

\goodbreak
\begin{lemma}
\label{lem:w-delta-asymptotics}
For every $\Delta,r \in \N$, we have
\begin{enumerate}[label=(\roman*)]
 \item $\wv(\Delta,r) \leq 6 r \Delta^{\frac{1}{r}}$ and
 \item $\w(\Delta,r) \leq 3 r \Delta^{\frac{1}{r}}$.
\end{enumerate}
Moreover, we have $\wv(\Delta,r) \in \Theta(\Delta^{\frac{1}{r}})$ and $\w(\Delta,r) \in \Theta(\Delta^{\frac{1}{r}})$.
\end{lemma}

\begin{proof}
(i): Let $S \subseteq \Z^r$ be a point set in general position and with simplex-volume bounded by~$\Delta$.
Let $T \subseteq P_S$ be a simplex of maximal volume contained in~$P_S$.
It is known that~$T$ can be chosen to have all its vertices as vertices of~$P_S$, and by a result of Lagarias \& Ziegler~\cite[Thm.~3]{lagariasziegler1991bounds}, there is a translate of $(-r)T$ that contains~$P_S$.

Since we assumed~$S$ to have simplex-volume bounded by~$\Delta$, we have that $\Vol(T) \leq \Delta$, and thus $r!\vol(P_S) = \Vol(P_S) \leq \Vol((-r)T) = r^r \Vol(T) \leq r^r \Delta$.
Using the asymptotic result~\eqref{eqn:makai-conj} for the lattice polytope~$P_S$, yields
\[
\omega_L(P_S)^r \leq \left(\frac{8}{\pi}\right)^r \frac{2^r r!}{r+1}  \vol(P_S) \leq \left(\frac{8}{\pi}\right)^r \frac{2^r r^r}{r+1} \Delta,
\]
and thus
\[
\omega_L(P_S) \leq \frac{16}{\pi} \frac{r}{(r+1)^{\frac{1}{r}}} \Delta^{\frac{1}{r}} \leq 6 r \Delta^{\frac{1}{r}}.
\]
(ii): Let $A \subseteq \Z^r$ be a generic $\Delta$-modular set of lattice points.
Again we argue similarly as in part (i) and aim to estimate the volume of~$P_A$.
In contrast to the situation of sets in general position, we need to make a detour via the symmetric polytope $Q_A := P_{A \cup (-A)}$ spanned by the points in~$A$ and their negatives.
Clearly, $P_A \subseteq Q_A$ and thus $\omega_L(P_A) \leq \omega_L(Q_A)$.

Now, let $C \subseteq Q_A$ be a crosspolytope of maximal volume contained in~$Q_A$.
Again we can assume that $C$ has all its vertices as vertices of~$Q_A$, and in particular, we can write $C = \conv\{\pm a_1,\ldots,\pm a_r\}$ for some linearly independent $a_1,\ldots,a_r \in A$.
Note that the volume of~$C$ is given by $\vol(C) = \frac{2^r}{r!}\abs{\det(a_1,\ldots,a_r)}$.
Consider the parallelepiped $W = \left\{\sum_{i=1}^r \gamma_i a_i : -1 \leq \gamma_i \leq 1, \textrm{ for } 1 \leq i \leq r \right\}$.

We claim that $Q_A \subseteq W$.
Indeed, if we assume the contrary, then we find a vertex $v$ of~$Q_A$, which will be a point of~$A$ or~$-A$, such that $v \notin W$.
As the $a_1,\ldots,a_r$ are linearly independent, there are unique coefficients $\beta_1,\ldots,\beta_r \in \R$ such that $v = \beta_1 a_1 + \ldots + \beta_r a_r$, and there must be some index~$j$ such that $\abs{\beta_j} > 1$.
As a consequence we get
\begin{align*}
\abs{\det(a_1,\ldots,a_{j-1},v,a_{j+1},\ldots,a_r)} &= \abs{\det(a_1,\ldots,a_{j-1},\beta_j a_j,a_{j+1},\ldots,a_r)} \\
&= \abs{\beta_j} \abs{\det(a_1,\ldots,a_{j-1},a_j,a_{j+1},\ldots,a_r)} \\
&> \abs{\det(a_1,\ldots,a_r)}.
\end{align*}
This means that the crosspolytope $\conv\{\pm a_1,\ldots,\pm a_{j-1},\pm v,\pm a_{j+1},\ldots,\pm a_r\} \subseteq Q_A$ has larger volume than~$C$, contradicting the maximality of the latter.

Now, since $A$ is $\Delta$-modular, we get $\vol(Q_A) \leq \vol(W) = 2^r \abs{\det(a_1,\ldots,a_r)} \leq 2^r \Delta$.
Applying the inequality~\eqref{eqn:makai-conj-symmetric} for the symmetric lattice polytope~$Q_A$, yields
\[
\omega_L(Q_A)^r \leq \left(\frac{4}{\pi}\right)^r r! \vol(Q_A) \leq \left(\frac{4}{\pi}\right)^r 2^r r! \Delta,
\]
and thus
\[
\omega_L(P_A) \leq \omega_L(Q_A) \leq \frac{8}{\pi} (r!)^{\frac{1}{r}} \Delta^{\frac{1}{r}} \leq 3 r \Delta^{\frac{1}{r}}.
\]
This finishes the proof of (ii).

To see why the previously derived upper bounds on $\wv(\Delta,r)$ and $\w(\Delta,r)$ are asymptotically best possible, we employ a scaling argument.
We claim that for every $\ell \in \Z_{>0}$,
\begin{align}
\ell \cdot \wv(\Delta,r) \leq \wv(\ell^r \Delta,r) \qquad \textrm{ and } \qquad \ell \cdot \w(\Delta,r) \leq \w(\ell^r \Delta,r).\label{eqn:scaling-wv-and-w}
\end{align}
We give the argument for $\wv(\Delta,r)$, since the one for $\w(\Delta,r)$ is analogous.
Let $S \subseteq \Z^r$ be a set in general position with simplex-volume at most~$\Delta$ and such that $\omega_L(P_S) = \wv(\Delta,r)$.
Now, $\ell S = \left\{ \ell s : s \in S \right\} \subseteq \Z^r$ is also in general position and has simplex-volume at most~$\ell^r \Delta$.
Moreover, $P_{\ell S} = \conv(\ell S) = \ell \conv(S) = \ell P_S$, and thus
\[
\wv(\ell^r \Delta,r) \geq \omega_L(P_{\ell S}) = \omega_L(\ell P_S) = \ell \omega_L(P_S) = \ell \wv(\Delta,r).
\]
Now, assume to the contrary that, say $\wv(\Delta,r) \in o(\Delta^{\frac{1}{r}})$.
In view of~\eqref{eqn:scaling-wv-and-w}, this implies that $\ell \wv(\Delta,r) \in o((\ell^r \Delta)^{\frac{1}{r}}) = o(\ell) \cdot o(\Delta^{\frac{1}{r}})$, which is a contradiction for $\ell \to \infty$.
\end{proof}

\begin{remark}
The previous arguments for the upper bound on~$\w(\Delta,r)$ do not use the genericity of the $\Delta$-modular subsets $A \subseteq \Z^r$ defining this quantity.
The argument for $\w(\Delta,r) \in \Omega(\Delta^{\frac{1}{r}})$ works both for the generic case, as well as the non-generic situation.
However, it is not clear whether $\w(\Delta,r)$ is actually given as the maximal lattice-width of a polytope $P_A$, for some $\Delta$-modular set $A \subseteq \Z^r$.
\end{remark}

We proceed by an elementary estimate on the type of sums occurring in \cref{lem:slice-reduction}.

\begin{lemma}
\label{lem:riemann-sum}
For every integers $a, n, r$ with $r \geq 2$, we have
\[
\sum_{\ell=a}^{a+n} \abs{\ell}_+^{-\frac{1}{r}} \leq 1+4 n ^{\frac{r-1}{r}}.
\]
\end{lemma}

\begin{proof}
As $\ell \mapsto 1/\abs{\ell}_+$ is an even function and decreases for $\abs{\ell}\to\infty$, the sum is maximal if the interval $[a, a+n]$ is symmetric around 0.
Therefore, we have 
\[
\sum_{\ell=a}^{a+n} \abs{\ell}_+^{-\frac{1}{r}}
    \leq 1 + 2 \sum_{\ell=1}^{\ceil{n/2}} \abs{\ell}_+^{-\frac{1}{r}} = 1 + 2 \sum_{\ell=1}^{\ceil{n/2}}\ell^{-\frac{1}{r}}.
\]
The sum in the term on the right hand side is a lower Riemann sum for the integral $\int_0^{\ceil{n/2}} x^{-1/r} \dif x$.
Hence, we obtain
\[
\sum_{\ell=1}^{\ceil{n/2}} \ell^{-\frac{1}{r}}
\leq \int_0^{\ceil{n/2}} x^{-1/r} \dif x
= \frac{r}{r-1} \ceil{\frac{n}{2}}^{\frac{r-1}{r}}
\leq 2 n ^{\frac{r-1}{r}},
\]
which implies the claim.
\end{proof}

We are now prepared to prove the main result of this section.

\begin{proof}[Proof of \cref{thm:sublinear-bound-generic-heller}]
We just combine \cref{lem:slice-reduction} with \cref{lem:w-delta-asymptotics,lem:riemann-sum}.
More precisely, we get for every $r \geq 3$ that
\begin{align*}
  \g(\Delta,r) 
    &\leq \sup_{a \in \Z} 
      \sum_{\ell = a}^{a+ \w(\Delta,r)} 
        \gv\left(\frac{\Delta}{\abs{\ell}_+},r-1\right) \\
    &\leq \sup_{a\in\Z} 
      \sum_{\ell = a}^{a+ \w(\Delta,r)}
        \left(\wv\left(\frac{\Delta}{\abs{\ell}_+},r-1\right) + 1\right) \cdot (r-1) \\
    &\leq (r-1)\left(\w(\Delta, r)+1 + 
      \sup_{a\in\Z}\sum_{\ell=a}^{a+\w(\Delta, r)} 6 (r-1) \left(\frac{\Delta}{\abs{\ell}_+}\right)^{\frac{1}{r-1}}\right) \\
    &\leq (r-1)\left(\w(\Delta, r)+1 + 6(r-1)\Delta^{\frac{1}{r-1}}
      \sup_{a\in\Z}\sum_{\ell=a}^{a+\w(\Delta, r)} \left(\frac{1}{\abs{\ell}_+}\right)^{\frac{1}{r-1}}\right) \\
    &\leq (r-1)\left(\w(\Delta, r)+1 + 6(r-1)\Delta^{\frac{1}{r-1}}
      \left(1+4 \w(\Delta, r)^{\frac{r-2}{r-1}}\right)\right) \\
    &\leq (r-1)\left(3r\Delta^{\frac{1}{r}}+1 + 6(r-1)\Delta^{\frac{1}{r-1}}
    \left(1+4 \left(3r\Delta^{\frac{1}{r}}\right)^{\frac{r-2}{r-1}}\right)\right) \\
    &\leq r \cdot 10r \Delta^{\frac{1}{r-1}} \cdot 13 r \Delta^{\frac{r-2}{r(r-1)}}\\
    &= 130\, r^3 \Delta^{\frac{1}{r-1}+\frac{r-2}{r(r-1)}} 
    = 130\, r^3 \Delta^{\frac{2}{r}}.\qedhere
\end{align*}
\end{proof}

\noindent Note that we didn't optimize the constant in the bound in the above proof, but rather aimed for a simple derivation.

\section{Generic \texorpdfstring{$\Delta$}{\textDelta}-modular matrices with many columns}
\label{sect:lower-bounds}

In this section, we discuss to which extent the upper bounds in \cref{thm:linear-upper-bound} and \cref{thm:sublinear-bound-generic-heller} are best possible.
We find that for $r=2$ rows, there are infinitely many values of~$\Delta$ for which the bound in \cref{thm:linear-upper-bound} is tight, while for $r \geq 3$, our best lower bound construction on $\g(\Delta,r)$ is of order $\Omega(\Delta^{\frac{1}{r-1}})$ in contrast to the bound $\cO(\Delta^{\frac{2}{r}})$ in \cref{thm:sublinear-bound-generic-heller}.

\subsection{Constructions for two rows}
\label{sect:constr-two-rows}

By definition, we have $\s(\Delta,2) = \g(\Delta,2)$, so that in view of~\eqref{eqn:s-facet-number} the considerations in this part concern the maximum number of irredundant inequalities in a $\Delta$-modular integer program with two variables.
We keep the focus on $\g(\Delta,2)$ though, and use this notation for the same number throughout.

If $r=2$, then the upper bound in \cref{thm:linear-upper-bound} reads $\g(\Delta,2) \leq p+1$, where $p$ is the smallest prime larger than~$\Delta$.
In the following, we describe three infinite families for which this upper bound is attained:

\begin{enumerate}[label=(F\arabic*)]
  \item\label{itm:F1} $\g(\Delta,2) = \Delta + 2$, if $\Delta$ is even and $\Delta+1$ is prime
  \item\label{itm:F2} $\g(\Delta,2) = \Delta + 3$, if $\Delta \geq 3$ is odd and $\Delta + 2$ is prime
  \item\label{itm:F3} $\g(\Delta,2) = \Delta + 4$, if $\Delta \geq 4$ is even, $\Delta = 2 \bmod 3$, and $\Delta + 3$ is prime
\end{enumerate}
Note that there are values $\Delta$ for which $\g(\Delta, 2)$ does not meet the upper bound $p+1$.
Indeed, the smallest~$\Delta \geq 2$ that does not belong to any of the families \ref{itm:F1}--\ref{itm:F3} is $\Delta = 7$, and our computational experiments, that we describe in \cref{sect:numerical-results} below, show that $\g(7,2) = 10$, while $p+1 = 12$ in this case (see \cref{tab:gDelta2Values}).

Interestingly, our computational experiments in \cref{sect:numerical-results} also show that $\g(\Delta,2) \leq \Delta + 6$, for every $\Delta \leq 450$, so that the following problem may have a positive answer:

\begin{conjecture}
\label{conj:constant-excess}
Is there a constant $c>0$, such that $\g(\Delta,2) \leq \Delta + c$, for all $\Delta > 0$?
\end{conjecture}
The weaker question whether $\limsup_{\Delta \to \infty} \frac{\g(\Delta,2)}{\Delta} = 1$ can be answered in the affirmative by using any of the families~\ref{itm:F1}--\ref{itm:F3}, and in view of the results in~\cite{bakerharmanpintz2001thedifference}.
Therein it is shown that for sufficiently large $x$, the interval $[x - x^{0.525},x]$ contains at least one prime, which implies $\g(\Delta,2) \leq \Delta + \cO(\Delta^{0.525})$.

We now describe matrices that attain the claimed identities for the families \ref{itm:F1}--\ref{itm:F3}.

\subsubsection*{Matrices for \ref{itm:F1} and \ref{itm:F2}}

A construction showing $\g(\Delta,2) \geq \Delta + 2$, for \emph{every} $\Delta \in \Z_{>0}$, is mentioned in~\cite{paatstallknechtwalshxu2022onthecolumn}.
Indeed, the matrix with $\Delta+2$ columns given by
\begin{align}
 \begin{pmatrix}
   1 & 0 & 1 & 1 & \cdots & 1      \\
   0 & 1 & 1 & 2 & \cdots & \Delta
 \end{pmatrix}
 \label{eqn:Delta+2Ex}
\end{align}
is generic $\Delta$-modular, and thus attains $\g(\Delta,2)$ for the values~$\Delta$ belonging to family~\ref{itm:F1}.

If $\Delta$ is odd, we can extend~\eqref{eqn:Delta+2Ex} and obtain a matrix with $\Delta+3$ columns given by 
\begin{align}
 \begin{pmatrix}
   1 & 0 & 1 & 1 & \cdots & 1      & 2 \\
   0 & 1 & 1 & 2 & \cdots & \Delta & \Delta
 \end{pmatrix},
 \label{eqn:Delta+3Ex}
\end{align}
which is again generic $\Delta$-modular.
This shows that $\g(\Delta, 2)\geq \Delta+3$, for \emph{every} odd~$\Delta$ and meets the upper bound for the family~\ref{itm:F2}.

\subsubsection*{Matrices for \ref{itm:F3}}

The construction for this family is more involved.
For $m \in \N$ and integer vectors $a,b \in \Z^m$ with $a_j \leq b_j$, for all $j = 1,\ldots,m$, we define the matrix
\[
M(a,b) :=   \left(
  \begin{array}{cccccccccccccc}
    0 &   1 & \cdots &   1 &   2 & \cdots &   2 &   3 & \cdots &   3 & \cdots &   m & \cdots &   m  \\
    1 & a_1 & \cdots & b_1 & a_2 & \cdots & b_2 & a_3 & \cdots & b_3 & \cdots & a_m & \cdots & b_m
  \end{array}
  \right),
\]
where, for $j = 1, \ldots, m$, the dots between $a_j$ and $b_j$ in the second row indicate the list of all integers $k$ such that $a_j \leq k \leq b_j$ and $\gcd(j,k) = 1$.
So, by construction the columns of $M(a,b)$ are pairwise non-parallel and thus the matrix is generic.
Observe that the matrices~\eqref{eqn:Delta+2Ex} and~\eqref{eqn:Delta+3Ex} also have this form.
Indeed, up to column permutations,~\eqref{eqn:Delta+2Ex} equals $M(0,\Delta)$, whereas~\eqref{eqn:Delta+3Ex} equals $M(\binom{0}{\Delta},\binom{\Delta}{\Delta})$.
With this notation, our task is to find vectors $a,b \in \Z^m$ such that $M(a,b)$ has many columns while the absolute value of the minors of size two remain bounded.

\begin{proposition}
\label{prop:r2-construction-family3}
Let $\Delta \geq 4$ be even and such that $\Delta = 2 \bmod 3$.
For the following vectors $a,b \in \Z^3$ the matrix $M(a,b)$ is $\Delta$-modular and has $\Delta + 4$ columns:
\begin{enumerate}[label=(\roman*)]
 \item If $\Delta = 12 s + 2$, for some $s \in \N$, let
 \[
 a = (0,4 s + 1, 9 s + 1)^\intercal \quad\textrm{ and }\quad b = (7 s + 1, 10 s + 1, 12 s + 2)^\intercal.
 \]
 \item If $\Delta = 12 s + 8$, for some $s \in \N$, let
 \[
 a = (0,4 s + 3, 9 s + 7)^\intercal \quad\textrm{ and }\quad b = (7 s + 5, 10 s + 7, 12 s + 8)^\intercal.
 \]
\end{enumerate}
\end{proposition}

\begin{proof}
(i): We first count the columns of
\[
M(a,b) =
\begin{pmatrix}
0 & 1 & \cdots &    1 &    2 & \cdots &     2 &    3 & \cdots &     3 \\
1 & 0 & \cdots & 7s+1 & 4s+1 & \cdots & 10s+1 & 9s+1 & \cdots & 12s+2 
\end{pmatrix}.
\]
There are $7s+2$ columns of the form $(1,k)^\intercal$, there are $3s+1$ columns of the form $(2,k)^\intercal$ corresponding to the odd numbers between $4s+1$ and $10s+1$, and there are $2s+2$ columns of the form $(3,k)^\intercal$ corresponding to the numbers between $9s+1$ and $12s+2$ that are not divisible by~$3$.
Together with the column $(0,1)^\intercal$ this gives in total $12s + 6 = \Delta + 4$ columns.

We now prove that all the minors of size two in $M(a,b)$ are bounded by~$\Delta$ in absolute value.
Clearly all determinants involving $(0,1)^\intercal$ satisfy the constraint. For all other determinants we may check the upper and lower bound separately:
\begin{align*}
  \abs{\det\begin{pmatrix}
    1   & 1     \\
    k_1 & k_1'
  \end{pmatrix}}
  &= \abs{k_1'-k_1} \leq 7s+1 \leq \Delta  \\
  \abs{\det\begin{pmatrix}
    2   & 2     \\
    k_2 & k_2'
  \end{pmatrix}}
  &= \abs{2(k_2'-k_2)} \leq 2\,(10s+1-(4s+1)) = 12s \leq \Delta \\
  \abs{\det\begin{pmatrix}
    3   & 3     \\
    k_3 & k_3'
  \end{pmatrix}}
  &= \abs{3(k_3'-k_3)} \leq 3\,(12s+2-(9s+1)) = 9s+3 \leq \Delta \\
  \det\begin{pmatrix}
    1   & 2     \\
    k_1 & k_2
  \end{pmatrix}
  &=k_2-2k_1 \leq 10s+1-2 \cdot 0 = 10s+1 \leq \Delta \\
  \det\begin{pmatrix}
    1   & 2     \\
    k_1 & k_2
  \end{pmatrix}
  &=k_2-2k_1 \geq 4s+1-2\,(7s+1) = -10s-1 \geq -\Delta \\
  \det\begin{pmatrix}
    1   & 3     \\
    k_1 & k_3
  \end{pmatrix}
  &=k_3-3k_1 \leq 12s+2-3 \cdot 0 = 12s+2 \leq \Delta \\
  \det\begin{pmatrix}
    1   & 3     \\
    k_1 & k_3
  \end{pmatrix}
  &=k_3-3k_1 \geq 9s+1-3\,(7s+1) = -12s-2 \geq -\Delta \\
  \det\begin{pmatrix}
    2   & 3     \\
    k_2 & k_3
  \end{pmatrix}
  &=2k_3-3k_2 \leq 2\,(12s+2)-3\,(4s+1) = 12s+1 \leq \Delta \\
  \det\begin{pmatrix}
    2   & 3     \\
    k_2 & k_3
  \end{pmatrix}
  &=2k_3-3k_2 \geq 2\,(9s+1)-3\,(10s+1) = -12s-1 \geq -\Delta. \\
\end{align*}
(ii): Here, $\Delta = 12 s + 8$, for some non-negative integer $s$, and we consider the matrix 
\[
M(a,b) =
\begin{pmatrix}
0 & 1 & \cdots &    1 &    2 & \cdots &     2 &    3 & \cdots &     3 \\
1 & 0 & \cdots & 7s+5 & 4s+3 & \cdots & 10s+7 & 9s+7 & \cdots & 12s+8 
\end{pmatrix}.
\]
Similarly to part~(i), we find that there are $7s+6$ columns of the form $(1,k)^\intercal$, there are $3s+3$ columns of the form $(2,k)^\intercal$, and there are $2s+2$ columns of the form $(3,k)^\intercal$.
Together with the column $(0,1)^\intercal$ this gives a total of $12s + 12 = \Delta + 4$ columns.

Showing that $M(a,b)$ is $\Delta$-modular is analogous to part~(i).
\end{proof}

\noindent Since the conditions on~$\Delta$ in \cref{prop:r2-construction-family3} imply that either $\Delta = 2 \bmod 12$ or $\Delta = 8 \bmod 12$, the matrices therein explain family~\ref{itm:F3}.

\subsection{A Vandermonde-type construction for an arbitrary number of rows}
\label{sect:vandermonde-construction}

The following bound subsumes the linear lower bound $\g(\Delta,2) \geq \Delta$, that is provided by the last~$\Delta$ columns of the matrix~\eqref{eqn:Delta+2Ex}, to a bound of order $\g(\Delta,r) \in \Omega(\Delta^{\frac{1}{r-1}})$, for every dimension~$r \geq 2$.

\begin{proposition}
\label{prop:generic-gen-heller-lower-bound}
For every $r \geq 2$ and any prime number $p$ with $p \geq r$, we have
\[
\g\left(\lceil (r-1)^{\frac{r-1}{2}} \rceil \cdot (p-1)^{r-1} , r \right) \geq p.
\]
\end{proposition}

\begin{proof}
For an integer $z \in \Z$, we let $[z]_p = z \bmod p$ be the representative in~$\{1,2,\ldots,p\}$.
With this notion, we consider the modular moment curve $\nu : \Z \to \{1,2,\ldots,p\}^{r-1}$ defined by $\nu(t) = ([t]_p, [t^2]_p , \ldots , [t^{r-1}]_p)$.
The~$p$ points $\nu(1),\nu(2),\ldots,\nu(p) \in \R^{r-1}$ are in general position, meaning that no~$r$ of them are contained in the same hyperplane in~$\R^{r-1}$ (cf.~\cite[\S 10.1]{brassmoserpach2005researchproblems}).
Lifting these points to~$r$ dimensions and arranging the~$p$ vectors $(1,\nu(1)),\ldots,(1,\nu(p))$ as columns of a matrix that we denote by~$A_{p,r} \in \Z^{r \times p}$, we find that~$A_{p,r}$ is generic with $\rank(A_{p,r}) = r$.
Further, by the Hadamard inequality for the determinant we have
\begin{align*}
\left|\det\left(
\begin{matrix}
1 & \ldots & 1 \\
\nu(\ell_1) & \ldots & \nu(\ell_r)
\end{matrix}
\right)\right|
&= \left|\det\left(
\begin{matrix}
1 & \ldots & 1 \\
[\ell_1]_p & \ldots & [\ell_r]_p \\
\vdots & \ddots & \vdots \\
[\ell_1^{r-1}]_p & \ldots & [\ell_r^{r-1}]_p
\end{matrix}
\right)\right| \\
&= \left|\det\left(
\begin{matrix}
1 & 1 & \ldots & 1 \\
0 & [\ell_2]_p - [\ell_1]_p & \ldots & [\ell_r]_p - [\ell_1]_p \\
\vdots & \vdots & \ddots & \vdots \\
0 & [\ell_2^{r-1}]_p - [\ell_1^{r-1}]_p & \ldots & [\ell_r^{r-1}]_p - [\ell_1^{r-1}]_p
\end{matrix}
\right)\right| \\
&\leq \prod_{i=2}^r \left\| \left( [\ell_i]_p - [\ell_1]_p,\ldots,[\ell_i^{r-1}]_p - [\ell_1^{r-1}]_p \right) \right\| \\
&\leq \left( \sqrt{r-1} \cdot (p-1)\right)^{r-1} ,
\end{align*}
for every $1 \leq \ell_1 < \ldots < \ell_r \leq p$.
In other words, the matrix~$A_{p,r}$ is $\Delta$-modular, where $\Delta = \lceil (r-1)^{\frac{r-1}{2}} \rceil \cdot (p-1)^{r-1}$, and the claimed bound follows.
\end{proof}
  

Given the upper bound on $\g(\Delta,r)$ of order $\cO(\Delta^{\frac{2}{r}})$ in \cref{thm:sublinear-bound-generic-heller}, and the lower bound of order $\Omega(\Delta^{\frac{1}{r-1}})$ above, the question about the precise asymptotic behavior of $\g(\Delta,r)$, for fixed~$r$, remains.
We tend to believe that the lower bound is the correct order of growth:

\begin{question}
For fixed $r \geq 3$, is it true that $\g(\Delta,r) \in \Theta(\Delta^{\frac{1}{r-1}})$?
\end{question}

\section{A computational approach for small parameters}
\label{sect:computational-approach}

Complementing the theoretical results on $\g(\Delta,r)$ concerning upper bounds in \cref{sect:upper-bounds} and lower bounds in \cref{sect:lower-bounds}, we discuss in the following a computational approach to algorithmically determine the precise values of $\g(\Delta,r)$ for small parameters~$\Delta$ and~$r$.
A related yet different computational approach for the ``non-generic'' numbers $\h(\Delta,r)$ has been described in~\cite{averkovschymura2023onthemaximal} (see the discussion in \cref{sect:non-generic-computations}).

We begin the section with some elementary properties of generic $\Delta$-modular matrices that take the symmetries of the problem into account, in particular its invariance under unimodular transformations.
Based on these properties we describe our algorithm in \cref{sect:algorithm-description} and investigate simple reductions that speed up the implementation considerably.
Finally, in \cref{sect:numerical-results}, we report on the numerical results of our computations and we aim to explain the findings at least for the rank two case as best as possible.

\subsection{Some properties of generic \texorpdfstring{$\Delta$}{\textDelta}-modular matrices}

We start with a linear algebra lemma about integer matrices.

\begin{lemma}
\label{lem:det-times-rational}
Let $A \in \Z^{r \times r}$ be invertible over~$\Q$ and let $S \in \Q^{r \times m}$.
If $AS$ is an integer matrix, then $\det(A) S$ is an integer matrix.
\end{lemma}

\begin{proof}
Let $\adj(A)$ be the adjugate matrix of $A$, which is an integer matrix because~$A$ is.
Then $\adj(A)A = \det(A) I$, and we write $C=AS$.
By assumption, $C$ is an integer matrix.
Multiplying both sides from the left by $\adj(A)$ we get $\adj(A) C= \adj (A) A S = \det(A) S$, which is an integer matrix because $\adj(A)$ and~$C$ are.
\end{proof}

\goodbreak
\begin{definition}
Let $A \in \Q^{r\times n}$ be a rational matrix with $\rank(A) = r$.
We call~$A$
\begin{enumerate}[label=(\roman*)]
 \item \emph{$\Delta$-bound}, if for $1\leq m\leq r$ every $m\times m$-minor is bounded by $\Delta^m$ in absolute value,
 \item \emph{totally generic}, if for $1\leq m\leq r$ every $m\times m$-minor is non-zero.
\end{enumerate}
\end{definition}
\noindent If $A$ is a totally generic $\Delta$-bound matrix, then in particular all entries of~$A$ are non-zero and bounded by~$\Delta$ in absolute value.
An equivalent condition to being $\Delta$-bound is that $A/\Delta$ is totally $1$-submodular.
As a sidenote, by Hadamard's inequality we have that if the absolute value of the entries of any matrix are bounded by~$\beta$, then each $m\times m$-minor of that matrix is bounded by $m^{m/2} \beta^m$ in absolute value.
So if $A$ is $\Delta$-bound this then gets rid of the factor~$m^{m/2}$.

A direct observation is that total $\Delta$-submodularity and total genericity reduce to $\Delta$-submodularity and genericity, respectively, if we amend a matrix with the identity matrix.

\begin{proposition}
\label{prop:totally-vs-simple-modularity-and-genericity}
Let $S \in \Q^{r \times n}$ have rank~$r$ and let $I \in \Z^{r \times r}$ be the identity matrix.
Then, for every $\Delta \in \Z_{>0}$, we have that
\[
S \textrm{ is totally } \Delta\textrm{-submodular} \quad \Longleftrightarrow \quad (I,S) \textrm{ is } \Delta\textrm{-submodular},
\]
and
\[
S \textrm{ is totally generic } \quad \Longleftrightarrow \quad (I,S) \textrm{ is generic}.
\]
\end{proposition}

\noindent We now develop the key lemma that is the basis of our algorithm to compute $\g(\Delta,r)$, for any given parameters~$\Delta,r$, in the next section. 
Note that our arguments are similar to those establishing the key lemma in~\cite[Lem.~1]{averkovschymura2023onthemaximal}.

It relies on the notion of a normal form for integer matrices under unimodular transformations.
We only need the concept for square matrices here.
For such an $A \in \Z^{r \times r}$, there exists a uniquely defined unimodular matrix $U \in \Z^{r \times r}$ such that $A = UH$ and $H \in \Z^{r \times r}$ is an upper triangular matrix with positive diagonal entries $h_{11},\ldots,h_{rr}$ and off-diagonal entries $0 \leq h_{ij} < h_{jj}$, for all $1 \leq j \leq r$ and $1 \leq i < j$.
The matrix~$H$ is called the \emph{Hermite normal form} of~$A$ (see~\cite[Ch.~2]{cohen1993acourse} for more background and computational approaches to efficiently compute~$H$ from a given matrix~$A$).

\begin{lemma}
\label{lem:existence-Deltabound-totally-gen}
Let $D\in\Z^{r\times(r+L)}$ be a generic $\Delta$-modular matrix with $\rank(D)=r$.
Then there exists a totally generic $\Delta$-bound matrix $C\in\Z^{r\times L}$.
Further, each column of $C$ is a solution to a linear system of equations $Ax=0 \bmod \Delta$ for an integer matrix $A\in \Z^{r\times r}$ in Hermite normal form satisfying $\det(A)=\Delta$.
\end{lemma}

\begin{proof}
Since $D$ is $\Delta$-modular, there exists a $r\times r$ submatrix of $D$ with determinant~$\pm \Delta$.
After rearranging the columns of $D$ and applying row operations (where we only scale by $\pm 1$), we can assume that $D$ has the form
\begin{align}
D = (A, u_1, u_2, \ldots, u_L),\label{eqn:D-after-reduction}
\end{align}
with columns $u_i\in \Z^r$ for $i=1, \ldots, L$ and $A$ being in Hermite normal form with determinant $\Delta$.
  
As $\det(A)=\Delta\neq 0$, $A$ has full rank.
Therefore each column $u_i$ is a linear combination of the columns of $A$, in other words for each $1\leq i\leq L$ there exists a rational vector $s_i\in\Q^r$ with $As_i = u_i$.
Let $S=(s_1, \ldots, s_L)$ and $C=\Delta S$.
By \cref{lem:det-times-rational}, $C$ is an integer matrix.
Also note that as $As_i=u_i$, we get that $Ac_i = A(\Delta s_i) = \Delta u_i = 0 \bmod \Delta$ for the columns $c_1,\ldots,c_L$ of $C$.

It remains to show that $C \in \Z^{r \times L}$ is $\Delta$-bound and totally generic.
Since $C = \Delta S$, it suffices to prove that $S$ is totally $1$-submodular and totally generic.
In order to see this, observe that by construction $D = A \cdot (I,S)$.
Because~$D$ is generic and $\Delta$-modular and $\det(A) = \Delta \neq 0$, this means that $(I,S)$ is generic and $1$-submodular.
Hence, in view of \cref{prop:totally-vs-simple-modularity-and-genericity}, $S$ is totally generic and totally $1$-submodular as claimed.
\end{proof}

It is known that the entries of the matrix~\eqref{eqn:D-after-reduction} in the proof of \cref{lem:existence-Deltabound-totally-gen} can be bounded by~$\Delta$ in absolute value (cf.~\cite[Lem.~1]{gribanovmalyshevpardalosveselov}).

\subsection{An algorithm to compute \texorpdfstring{$\g(\Delta,r)$}{g(\textDelta,r)}}
\label{sect:algorithm-description}

\cref{lem:existence-Deltabound-totally-gen} and its proof suggest an algorithm for finding a generic $\Delta$-modular matrix of size $r\times (r+L)$ with maximal $L$.

\begin{enumerate}[label=Step~\arabic*:]
  \item Construct all integer matrices $A\in \Z^{r\times r}$ in Hermite normal form with determinant $\det(A)=\Delta$.
  
  
  \item Find all solutions $x \in \Z_{\Delta}^r$ of $Ax=0 \bmod \Delta$. For each such solution $x$ calculate all possible representatives in $\Z^r$ with nonzero entries bounded by~$\Delta$ in absolute value.
  We can assume that the first entry is positive.
  
  
  \item Find a largest possible totally generic $\Delta$-bound matrix with the columns obtained in Step 2.
  
  
  \item Let $C \in \Z^{r \times L}$ be the totally generic $\Delta$-bound matrix found in Step 3.
  Then, $D=(A, AC/\Delta)$ is a generic $\Delta$-modular matrix of size $r\times(r+L)$.
\end{enumerate}
Pseudocode for this algorithmic approach is given as \cref{alg:computeHeller}.
Correctness of the algorithm follows directly from \cref{lem:existence-Deltabound-totally-gen}, so that in the remainder of this subsection we will be concerned with describing how we can implement the various steps in an efficient manner.

\begin{algorithm}[ht]
\caption{Compute $\g(\Delta, r)$}
\begin{algorithmic}[1]
  \Require $\Delta, r \in \Z_{>0}$
  \Ensure $\g(\Delta, r)$, generic $\Delta$-modular matrix $D \in \Z^{r \times \g(\Delta,r)}$
  \State $\g \gets 0$
  \State $D \gets 0$
  \State $H \gets $ list of all relevant Hermite normal forms\label{alg:hermite-step}
  \ForAll {$A\in H$}\label{alg:for-loop}
    \State $X \gets$ relevant integer solutions to $Ax = 0 \bmod \Delta$\label{alg:relevant-solutions}
    \State $C \gets$ largest possible totally generic $\Delta$-bound matrix with columns in $X$\label{alg:expensive-step}
    \If{$r+\operatorname{number-of-columns}(C) > \g$}
      \State $\g \gets r+\operatorname{number-of-columns}(C)$
      \State $D \gets (A, A C/\Delta)$
    \EndIf
  \EndFor
  \State \textbf{return} $\g, D$
\end{algorithmic}\label{alg:computeHeller}
\end{algorithm}

Note that Step 3 above, respectively Line~\ref{alg:expensive-step} in \cref{alg:computeHeller}, is by far the most computationally expensive step.
Therefore, we first investigate whether we actually need to run through \emph{all} the Hermite normal forms in Step 1, and likewise through \emph{all} integer solutions to the linear system in Step 2.
It turns out that we can reduce the possibilities quite a bit and only need to consider the \emph{relevant} of these objects as is hinted at in the description of \cref{alg:computeHeller} in Lines~\ref{alg:hermite-step} and~\ref{alg:relevant-solutions}.
Let us start with the Hermite normal forms.

\subsubsection{Relevant Hermite normal forms}

Among any pair~$A_1$ and~$A_2$ of integer matrices in Hermite normal form that give equivalent results, we need to iterate only over one of them.
Let us make more precise what we consider as equivalent results:

\begin{definition}
\label{def:equivalence}
Let $D_1, D_2 \in \Z^{r\times n}$.
We say that~$D_1$ is \emph{equivalent} to~$D_2$, in symbols $D_1 \simeq D_2$, if there exists a unimodular matrix $S \in \Z^{r \times r}$, a diagonal matrix $D \in \Z^{n \times n}$ with diagonal entries $\pm 1$, and a permutation matrix $P \in \Z^{n \times n}$ such that $D_2 = S D_1 D P$.
\end{definition}

Hence, $D_1 \simeq D_2$, if $D_2$ can be obtained from $D_1$ by applying integer row operations, rearranging columns and multiplying columns by~$-1$.
Note that both $\Delta$-modularity and genericity are preserved under this notion of equivalence.

Let $D_1=(A_1, U_1)$ be a generic $\Delta$-modular matrix with $A_1$ in Hermite normal form and let $A_2$ be a matrix in Hermite normal form and which is equivalent to $A_1$, i.e. $A_2 = S A_1 D P$ for suitable matrices $S,D$, and~$P$ as in \cref{def:equivalence}.
Then $D_2 = (S A_1 D P, S U_1) = (A_2, S U_1)$ is also generic $\Delta$-modular and will be found by the algorithm.
In particular, during the algorithm we do not have to run through the for loop for both~$A_1$ and~$A_2$.
However, even if $D_1 = (A_1, U_1), D_2 = (A_2, U_2)$ are matrices with $A_1, A_2$ in Hermite normal form and such that $D_1 \simeq D_2$, then~$A_1$ and~$A_2$ need not be equivalent.
There can be a submatrix $\tilde A_1$ of $U_1$ which is equivalent to $A_2$. Therefore, even if we restrict the equivalence to the Hermite normal forms that we investigate, and for each equivalence class we test only one matrix, we might still find resulting matrices during the algorithm that are equivalent.

This problem is related to finding a canonical form under the equivalence relation~$\simeq$.
We are not aware of any work in the literature solving this problem, but we discuss two operations below that can be used to reduce the amount of Hermite normal forms that have to be checked.
Note that a slightly different notion of equivalence has been studied by Paolini~\cite{paolini2017analgorithm}.
He says that two matrices $D,D' \in \Z^{r \times n}$ are equivalent, if there is a permutation matrix $P \in \Z^{n \times n}$ and an affine unimodular equivalence~$\varphi : \Z^r \to \Z^r$, that is $\varphi(x) = Ax + b$, for some unimodular $A \in \Z^{r \times r}$ and some $b \in \Z^r$, such that $D' = \varphi(D)P$.
Paolini defines a canonical form for this notion and devises an efficient algorithm to compute it.
The question on whether his approach can be adapted to our notion of equivalence will be left for future work.

\subsubsection*{Operation 1: Sorting the diagonal}

Let $A \in \Z^{r \times r}$ be a matrix in Hermite normal form with diagonal $(a_1, \ldots, a_r)$.
We claim that we can apply suitable column and row operations to obtain an equivalent matrix $A' \simeq A$ in Hermite normal form with diagonal $(a_1', \ldots, a_r')$ such that $a_1' \leq \ldots \leq a_r'$.
To see this it suffices to show how we may ``swap'' two adjacent diagonal entries, that is, if~$A$ has diagonal $(a_1, \ldots, a_{i-1}, a_i, a_{i+1}, a_{i+2}, \ldots, a_r)$ with $a_i > a_{i+1}$, then we can obtain an equivalent matrix $A'$ in Hermite normal form with the diagonal $(a_1, \ldots, a_{i-1}, a_{i+1}', a_i', a_{i+2}, \ldots, a_r)$ and $a_{i+1}' \leq a_i'$.
This can be achieved by following a simple procedure:
\begin{enumerate}[label=(\arabic*)]
 \item Write $b$ for the entry of~$A$ at position $(i,i+1)$ and exchange columns $i$ and $i+1$ of $A$.
 \item Perform the euclidean algorithm on $b$ and $a_{i+1}$ by applying row operations only involving rows~$i$ and~$i+1$ until the matrix is again upper triangular.
 \item Use rows $i,i+1,\ldots,r$ of the obtained matrix to restore the Hermite normal form condition for $i$-th row.
 \item Use rows $i+1,\ldots,r$ to restore the Hermite normal form condition for $(i+1)$-st row.
\end{enumerate}
The following scheme illustrates the procedure:
\begin{align*}
A &= \begin{pmatrix}
    a_1 \\
        & \ddots \\
        &        & a_i & b \\
        &        &   0 & a_{i+1} \\
        &        &     &         & \ddots \\
        &        &     &         &        & a_r
  \end{pmatrix}
  \rightarrow
  \begin{pmatrix}
    a_1 \\
        & \ddots \\
        &        & b       & a_i     \\
        &        & a_{i+1} & 0       \\
        &        &         &         & \ddots \\
        &        &         &         &        & a_r
  \end{pmatrix} \\
  &\rightarrow
  \begin{pmatrix}
    a_1 \\
        & \ddots \\
        &        & \gcd(b, a_{i+1}) & b'    \\
        &        &              0   & a_i'  \\
        &        &                  &         & \ddots \\
        &        &                  &         &        & a_r
  \end{pmatrix} = A'.
\end{align*}
Note that the entry of the obtained matrix~$A'$ at position $(i,i)$ is now $a_{i+1}'=\gcd(b, a_{i+1})\leq a_{i+1} < a_i$.
Since $A' \simeq A$, the determinant didn't change, and thus $a_i \cdot a_{i+1} = a_i' \cdot a_{i+1}'$.
Using $a_{i+1}'\leq a_{i+1}$, we obtain $a_i \leq a_i'$ and thus $a_{i+1}'\leq a_i'$, as desired.

In fact the described procedure shows that for $A = (a_{ij})$, we can assume that $a_{i,i}\leq \gcd(a_{i, i+1}, a_{i+1,i+1})$ for all $i=1, \ldots, r-1$ in addition to $a_{i,i}\leq a_{i+1, i+1}$.
It is possible that by exchanging also non-adjacent columns one may give a stronger condition based on greatest common divisors.
We have not looked into this further, though.

\subsubsection*{Operation 2: Sorting the last column}

Our second operation to reduce the number of relevant Hermite normal forms only applies to those with diagonal $(1, \ldots, 1, \Delta)$.
In this case, instead of having the last column be any of the vectors $(v_1, \ldots, v_{r-1}, \Delta)$ with $0\leq v_i\leq \Delta-1$ for $i=1,\ldots, r-1$, we can instead assume that $0 \leq v_1 \leq v_2 \leq \ldots \leq v_{r-1} \leq \Delta/2$.

Let $v_j > \Delta/2$ for some $j\in\{1, \ldots, r-1\}$.
We can multiply the $j$-th row by $-1$, then multiply the $j$-th column by $-1$, and then add the last row to the $j$-th row to obtain an equivalent matrix with identical entries except for the one in place of $v_j$ which is now $\Delta-v_j \leq \Delta/2$.
Note that we use here that the only nonzero elements in the $j$-th row and $j$-th column are the diagonal element and~$v_j$, otherwise we would have to do extra row additions to obtain a matrix in Hermite normal form again.

Next, to sort the last column, let $v_i > v_j$ with $i<j$.
We can simply exchange rows~$i$ and~$j$, and then exchange columns~$i$ and~$j$ to obtain an equivalent matrix in Hermite normal form with $v_i$ exchanged for~$v_j$.
Again we used the specific form of the original matrix here.

\medskip
Even in the case that $\Delta$ is prime -- which means that the only possible diagonal is $(1, \ldots, 1, \Delta)$ -- these two operations alone do not produce a canonical form.
As an example we observe that for $r=3$ and $\Delta=7$ the two matrices
\begin{equation*}
  A_1 = \begin{pmatrix}
    1 \\
     & 1 & 2 \\
     &   & 7
  \end{pmatrix}
  \qquad \text{and} \qquad
  A_2 = \begin{pmatrix}
    1 \\
     & 1 & 3 \\
     &   & 7
  \end{pmatrix}
\end{equation*}
are equivalent.
Indeed, by first multiplying the third column of $A_1$ by $-1$, then swapping the second and third column, and then computing the Hermite normal form of the resulting matrix, we obtain~$A_2$.

However, applying the two operations drastically reduces the number of Hermite normal forms that we need to consider in Line~\ref{alg:hermite-step} of \cref{alg:computeHeller}.
To exemplify this claim, let $H(\Delta,r)$ be the number of Hermite normal forms of size $r \times r$ with determinant~$\Delta$, let $H_{op}(\Delta,r)$ be the number of such Hermite normal forms that remain after applying Operations~1 and~2, and let $H_{\simeq}(\Delta,r)$ be the number of pairwise inequivalent such Hermite normal forms.
For comparison, in \cref{tbl:HNFs} we list these numbers for $r=4$ and some small values of~$\Delta$.

\begin{table}[ht]
  \begin{tabular}{c|*{30}{c}}
    $\Delta$             &    12 &    13 &    14 &    15 &    16 &    17 &    18 &    19 &    20     \\
    \hline
    $H(\Delta,4)$       &  6200 &  2380 &  6000 &  6240 & 11811 &  5220 & 18150 &  7240 & 24180     \\
    $H_{op}(\Delta,4)$        &   652 &    84 &   283 &   204 &  1267 &   165 &  1287 &   220 &  2523     \\
    $H_{\simeq}(\Delta,4)$ &   149 &    37 &    99 &    89 &   250 &    64 &   259 &    80 &   360     
  \end{tabular}
  \caption{A few numbers of Hermite normal forms under various reductions discussed in the text.}
  \label{tbl:HNFs}
\end{table}


The number $H(\Delta,r)$ is known exactly and depends on the prime decomposition $\Delta = p_1^{e_1} \cdot \ldots \cdot p_t^{e_t}$ of~$\Delta$ (cf.~\cite{gruber1997alternative} and the discussion therein for other formulae):
\begin{equation*}
H(\Delta,r) = \sum_{\substack{d_1,\ldots,d_r \in \Z_{>0}\\ d_1\cdot \ldots \cdot d_r = \Delta}} d_1^0 d_2^1 \cdot \ldots \cdot d_r^{r-1}  = \prod_{i=1}^t \prod_{j=1}^{e_i} \frac{p_i^{j+r-1} - 1}{p_i^j - 1} .
\end{equation*}

We are not aware of a similarly good understanding of the number of inequivalent Hermite normal forms:

\begin{question}
Is there a closed formula for $H_{\simeq}(\Delta,r)$ as well?
What is the growth rate of the numbers $H_{\simeq}(\Delta,r)$ ?
\end{question}

\subsubsection{Relevant integer solutions}

Now, we investigate the \emph{relevant} solutions to the equation in Line~\ref{alg:relevant-solutions} of \cref{alg:computeHeller}.
For a matrix $A \in \Z^{r \times r}$ in Hermite normal form and with $\det(A) = \Delta$, we therefore consider the linear system $Ax = 0 \bmod \Delta$.

Since~$A$ is in Hermite normal form, this system can be solved by backwards substitution, but as $\det(A) = \Delta$, the matrix~$A$ is not invertible over~$\Z_\Delta$.
Thus, there will be multiple solutions in~$\Z_\Delta^r$ and for each of those there might be several integer representatives.

As a first step, we determine the number of solutions of the linear system above in~$\Z_\Delta^r$.
For this, we need the Smith normal form of an integer matrix (cf.~Cohen~\cite[Ch.~2]{cohen1993acourse}).

\begin{theorem}
\label{thm:SNF}
Let $A \in \Z^{r \times r}$ have non-zero determinant.
Then, there exist unimodular matrices $Q, P \in \Z^{r \times r}$ such that $PAQ = \diag(\alpha_1, \ldots, \alpha_r)$, where the integers $\alpha_1,\ldots,\alpha_r$ are the elementary divisors of~$A$ satisfying $0 \leq \alpha_i | \alpha_{i+1}$, for all $1 \leq i \leq r-1$.
\end{theorem}

The matrix $S := PAQ$ in this theorem is uniquely determined and is called the \emph{Smith normal form} of~$A$.
By definition, we have $\abs{\det(A)} = \det(S) = \alpha_1 \cdot \ldots \cdot \alpha_r$.

\begin{lemma}
\label{lem:number-solutions-linear-system}
Let $a,b,\Delta \in \Z$ be integers with $\Delta > 0$ and let $A \in \Z^{r \times r}$ have determinant $\det(A)=\Delta$.
\begin{enumerate}[label=(\roman*)]
 \item The equation $ax = b \bmod \Delta$ has either $0$ or $\gcd(a, \Delta)$ solutions in~$\Z_\Delta$.
 \item There are exactly $\Delta$ solutions to $Ax = 0 \bmod \Delta$ in $\Z_\Delta^r$.
\end{enumerate}
\end{lemma}

\begin{proof}
(i): By linearity we only need to consider $b=0$.
If $ax=0 \bmod \Delta$, then $ax=k\Delta$ for some integer $k$. Let $d=\gcd(a, \Delta)$ and $a=ds, \Delta=dt$ be with $\gcd(s,t)=1$.
Plugging in, we obtain $dsx=kdt$, or equivalently $sx=kt$.
As $s$ divides $sx=kt$, but $\gcd(s,t)=1$, we have $x = (k/s)t$, where $k/s=m$ is an integer.
It is easy to check that indeed for any integer~$m$ the number $x=mt$ is a solution.
We now need to check that these give~$d$ distinct solutions in $\Z_\Delta$.

To this end, let $m_1, m_2\in \Z$ be with $m_1t=m_2t \bmod \Delta$.
Then, $(m_1-m_2)t = k\Delta = k d t$ for some integer $k$.
Therefore, $m_1-m_2=kd = 0\bmod d$ and $m_1=m_2 \bmod d$. Hence, each $0\leq m < d$ gives a distinct solution and this proves the claim.

(ii): Let $S=PAQ$ be the Smith normal form of~$A$.
Since as integer matrices $\abs{\det(P)} = \abs{\det(Q)} = 1$, both matrices~$Q,P$ are invertible over~$\Z_\Delta^r$.
Further, we have $\alpha_1 \cdot \ldots \cdot \alpha_r = \Delta$ for the diagonal elements of~$S$.

Multiplying $Ax = 0 \bmod \Delta$ from the left by~$P$ gives us
\[
PAx = PAQQ^{-1}x = S(Q^{-1}x) = 0 \bmod \Delta.
\]
As $P$ and $Q$ are invertible, the number of solutions to $Ax = 0 \bmod \Delta$ and $Sy = 0 \bmod \Delta$ with $y = Q^{-1}x$ are the same.
The $i$-th row of the system $Sy = 0 \bmod \Delta$ is given by $\alpha_i y_i = 0 \bmod \Delta$, which by part~(i) has exactly~$\alpha_i$ solutions in~$\Z_\Delta$, since~$\alpha_i$ is a divisor of~$\Delta$.
Combining these solutions for each row, we obtain exactly $\alpha_1 \cdot \ldots \cdot \alpha_r = \Delta$ solutions of $Ax = 0 \bmod \Delta$ over~$\Z_\Delta^r$.
\end{proof}

Now, as the second step, for each of the~$\Delta$ solutions $x \in \Z_\Delta^r$ to $Ax = 0 \bmod \Delta$, we want to find all the integer representatives in $x + \Delta \Z^r$ that are relevant for the computation of $\g(\Delta,r)$ by \cref{alg:computeHeller}.
Since we want to build a totally generic $\Delta$-bound matrix out of these vectors, in particular each of the entries of such a representative needs to be non-zero and bounded by~$\Delta$ in absolute value.
Further, if we find a matrix using the representative~$v$, then we will also find an equivalent matrix using the representative~$-v$, and no matrix can contain both of~$v$ and $-v$ because this would create a vanishing minor.
Therefore, we only need to take one of~$v$ and~$-v$, which we can achieve, for example, by fixing the first entry of each representative to be positive. 

In summary, if $x\in\Z_\Delta^r$ is a solution and $v \in x + \Delta \Z^r$ a representative under the discussed conditions, we have that if $x_i = 0$, for some index~$i$, then $v_i \in \{\Delta, -\Delta\}$, and if $x_i = k \neq 0$, then $v_i \in \{k, k-\Delta\}$.
This gives~$2$ options for each entry of~$v$, except for the first entry, which must be positive.
Hence, that leaves us with~$2^{r-1}$ relevant representatives of~$x$.
The list~$X$ in Line~\ref{alg:relevant-solutions} of \cref{alg:computeHeller} will therefore in total have $2^{r-1}\Delta$ vectors.

For particular matrices~$A$, we may further remove irrelavant representatives.
Indeed, if~$v$ and~$\ell v$ are elements in the list~$X$, for some $\ell > 1$, and we find a matrix using the column~$\ell v$, then we will also find a matrix using the column~$v$.
This means that we only need to take~$v$, and can neglect~$\ell v$.

\begin{remark}
The previous observations show that the matrices~$C$ that are constructed in \cref{alg:computeHeller} have at most~$2^{r-1}\Delta$ columns, which means that the generic $\Delta$-modular matrices $(A, AC/\Delta)$ have at most $r+2^{r-1}\Delta$ columns.
This proves $\g(\Delta, r) \in \cO(\Delta)$, for fixed values of~$r$, without using the theory of finite fields employed in \cref{sect:mds}.
However, here the implied constant depends exponentially on~$r$ instead of being constant.
Also, as discussed after \cref{thm:ball-mds}, the linear bound $\g(\Delta, r)\leq r+p-1$ is not very involved, and in particular is independent of Ball's result \cref{thm:ball-mds}.

A similar argument was carried out in~\cite[Rem.~1]{averkovschymura2023onthemaximal}, showing the bound $\h(\Delta,r) \leq 3^r \Delta$ for the maximal column number of not necessarily generic $\Delta$-modular matrices with~$r$ rows.
\end{remark}

\subsection{Finding a totally generic \texorpdfstring{$\Delta$}{Delta}-bound matrix with the most columns}

After the previous reductions have been carried out and, for a given Hermite normal from~$A$, the candidate set~$X$ has been computed, how can we now find a largest possible totally generic $\Delta$-bound matrix~$C$ with columns in~$X$?
We solve this problem by translating it into a clique problem in a suitably defined (hyper-)graph.

Recall that a clique in a simple graph $\cG = (V, E)$ is a collection of vertices $K \subseteq V$, so that for every $i, j \in K$ with $i \neq j$, we have $\{i,j\} \in E$.
A clique is called \emph{maximal} if it is not a proper subset of another clique, and it is called \emph{maximum} if there is no clique of larger cardinality in~$\cG$.
Now, let $\cH = (V,H)$ be a simple hypergraph, where $H \subseteq 2^V$ is the set of hyperedges.
For $k \geq 2$, a \emph{$k$-hyperclique} in~$\cH$ is a set $K \subseteq V$ of vertices so that every subset $I \subseteq K$ of size $1 < |I| \leq k$ is a hyperedge.
Maximal and maximum $k$-hypercliques are defined analogously to maximal and maximum cliques in graphs, respectively.

Now, assume we are given a list $X = \{v_1, \ldots, v_n\}$ of vectors with $n \leq 2^{r-1} \Delta$ generated in Line~\ref{alg:relevant-solutions} of \cref{alg:computeHeller}, and we want to build a maximum totally generic $\Delta$-bound matrix~$C$ with columns from~$X$.
We define the hypergraph $\cH_r = ([n],H_r)$, where $I \subseteq [n]$ is a hyperedge, if and only if $1<|I|\leq r$ and the matrix $(v_i : i \in I)$ is totally generic $\Delta$-bound.
Note that the hypergraph~$\cH_r$ is downward-closed, because if a matrix~$M$ is totally generic $\Delta$-bound, then so is any submatrix of~$M$ obtained by removing columns.
The problem of finding the matrix~$C$ above is now equivalent to finding a maximum $r$-hyperclique in~$\cH_r$.

Let us first consider the case $r=2$, in which the hypergraph $\cH_2$ is just a simple graph.
For the maximum clique problem in graphs there are various software implementations available.
We used the Sagemath~\cite{sagemath} implementation \verb|sage.graphs.cliquer.max_clique|.
For $r \geq 3$, we implemented an algorithm for the hypergraph clique problem in \CC, based on \verb|hClique| described in~\cite{torres2017hclique}.

\subsection{Numerical results}
\label{sect:numerical-results}

In this last part, we discuss the computational results that we obtained by implementing the algorithm described in the previous sections and running it for small parameters~$r$ and~$\Delta$.
The \verb|sage| and \CC\ source code and the computed data can be found at \url{https://github.com/BKriepke/DeltaModular}.

\subsubsection{The case \texorpdfstring{$r=2$}{r=2}}

Our implementation allows to compute the numbers $\g(\Delta,2)$, for every $2 \leq \Delta \leq 450$ in a reasonable amount of time.
In the following, we try to explain the obtained values as much as possible and make some hypotheses that they suggest.

We refrain from writing down each particular value of $\g(\Delta,2)$, for $\Delta \leq 450$, and instead refer the interested reader to the data in the repository linked above.
Investigating the data, the first observation that comes to mind is that $\g(\Delta,2)$ seems to be a non-decreasing function of~$\Delta$.
Up to now, we could not find a theoretical explanation.

\begin{conjecture}
\label{conj:monotonicity-two-rows}
The function $\Delta \mapsto \g(\Delta,2)$ is non-decreasing.
\end{conjecture}

In view of the matrix in~\eqref{eqn:Delta+2Ex} and \cref{thm:linear-upper-bound}, we have $\Delta + 2 \leq \g(\Delta,2) \leq p + 1$, with~$p$ being the smallest prime larger than~$\Delta$.
Let us first have a look at those $\Delta \leq 450$ for which this upper bound is attained.
A large part of these values is of course explained by the three infinite families~\ref{itm:F1}--\ref{itm:F3} from \cref{sect:constr-two-rows}.
\cref{tab:gDelta2Values} shows the results for the remaining parameters~$\Delta$ in the range $7 \leq \Delta \leq 124$, compared to the upper bound~$p+1$.

\begin{table}[ht]
  \begin{tabular}{c|*{30}{c}}
    $\Delta$       &  7 & 13 & 19 & 23 & 24 & 25 & 31 & 32 & 33 & 34 & 37 & 43 & 47 &   \\
    $\g(\Delta, 2)$ & 10 & 16 & 23 & 27 & \textbf{30} & \textbf{30} & 34 & 36 & 37 & \textbf{38} & \textbf{42} & 46 & 51 &   \\
    $p+1$          & 12 & 18 & 24 & 30 & \textbf{30} & \textbf{30} & 38 & 38 & 38 & \textbf{38} & \textbf{42} & 48 & 54 &   \\
    \hline
    $\Delta$       & 48 & 49 & 53 & 54 & 55 & 61 & 62 & 63 & 64 & 67 & 73 & 74 & 75 &   \\
    $\g(\Delta, 2)$ & \textbf{54} & \textbf{54} & 57 & 59 & \textbf{60} & 65 & 67 & 67 & \textbf{68} & 70 & 76 & 78 & 78 &   \\
    $p+1$          & \textbf{54} & \textbf{54} & 60 & 60 & \textbf{60} & 68 & 68 & 68 & \textbf{68} & 72 & 80 & 80 & 80 &   \\
    \hline
    $\Delta$       & 76 & 79 & 83 & 84 & 85 & 89 & 90 & 91 & 92 & 93 & 94 & 97 & 103 &   \\
    $\g(\Delta, 2)$ & \textbf{80} & 82 & 87 & 89 & 89 & 93 & 94 & 94 & 96 & 96 & \textbf{98} & 100 & 106 &   \\
    $p+1$          & \textbf{80} & 84 & 90 & 90 & 90 & 98 & 98 & 98 & 98 & 98 & \textbf{98} & 102 & 108 &   \\
    \hline
    $\Delta$       & 109 & 113 & 114 & 115 & 116 & 117 & 118 & 119 & 120 & 121 & 122 & 123 & 124 &   \\
    $\g(\Delta, 2)$ & 112 & 117 & 119 & 119 & 120 & 120 & 122 & 123 & 126 & 126 & 126 & 126 & \textbf{128} &   \\
    $p+1$          & 114 & 128 & 128 & 128 & 128 & 128 & 128 & 128 & 128 & 128 & 128 & 128 & \textbf{128} &   \\
    \end{tabular}
  \caption{Values for $\g(\Delta,2)$ compared with the upper bound $p+1$ from \cref{thm:linear-upper-bound}, for those $\Delta \leq 124$ that do not belong to any of the families~\ref{itm:F1}--\ref{itm:F3} from \cref{sect:constr-two-rows}. In bold are those values for which the upper bound is met.}
  \label{tab:gDelta2Values}
\end{table}

There are some further values of~$\Delta$ for which we can give an explicit construction that gives the correct value of~$\g(\Delta,2)$, even though those values do not attain the upper bound of \cref{thm:linear-upper-bound} (except for $\Delta=24$).
Indeed, for $\Delta = 30s + 24$, with~$s$ an integer, consider the matrix
\begin{align}
M(a,b) =  \left(\begin{matrix}
    0 & 1 & \cdots &     1 &    2 & \cdots &      2 &      3 & \cdots &      3    \\
    1 & 0 & \cdots & 11s+9 & 6s+5 & \cdots & 16s+13 & 12s+10 & \cdots & 21s+17 
  \end{matrix}\right. \notag\\
  \phantom{++++}
  \left.\begin{matrix}
         4 & \cdots &      4            &      5 & \cdots &      5   \\
    18s+15 & \cdots & 26s+13+\nu_s & 25s+21 & \cdots & 30s+24
  \end{matrix}\right),\label{eqn:Delta24mod30}
\end{align}
where $\nu_s$ is a suitable integer with $0 \leq \nu_s \leq 8 - s/2$.
This construction has the maximal number $\g(30s+24, 2) = \Delta+2+\floor{\nu_s/2}$ of columns; see \cref{tab:nuValues} for the precise values of the parameters~$\nu_s$.
In particular, for $s\leq 12$, this construction has more columns than the one given by~\eqref{eqn:Delta+2Ex}.
For $s \in \{13, 14\}$, we have $\g(\Delta, 2)=\Delta+2$,
which further indicates that eventually all even values~$\Delta$ either satisfy $\g(\Delta,2)=\Delta+4$, if $\Delta = 2 \bmod 3$, or $\g(\Delta,2)=\Delta+2$ otherwise.

\begin{table}[ht]
  \begin{tabular}{c|*{30}{c}}
    $s$                        & 0 & 1 & 2 & 3 & 4 & 5 & 6 & 7 & 8 & 9 & 10 & 11 & 12 \\
    $\Delta=30s+24$            & 24 & 54 & 84 & 114 & 144 & 174 & 204 & 234 & 264 & 294 & 324 & 354 & 384 \\
    $\nu_s$                    & 8 & 6 & 6 & 6 & 6 & 4 & 4 & 4 & 4 & 2 & 2 & 2 & 2 \\
    $\Delta+2+\floor{\nu_s/2}$ & 30 & 59 & 89 & 119 & 149 & 178 & 208 & 238 & 268 & 297 & 327 & 357 & 387 
  \end{tabular}
  \caption{The values $\nu_s$ for the construction in \eqref{eqn:Delta24mod30} realizing $\g(\Delta,2)$.}
 \label{tab:nuValues}
\end{table}


Let us now have a look at the excess $\varepsilon_\Delta := \g(\Delta,2) - (\Delta+2) \geq 0$ of $\g(\Delta,2)$ over the lower bound that is implied by~\eqref{eqn:Delta+2Ex}.
Already written down as \cref{conj:constant-excess} in \cref{sect:constr-two-rows}, our data suggests that~$\varepsilon_\Delta$ might be upper bounded by a constant.
In \cref{fig:gDelta2Values}, we have plotted~$\varepsilon_\Delta$ against $2 \leq \Delta \leq 450$, and we make the following particular observations explaining most of the values:

\begin{itemize}
  \item We have $\g(\Delta, 2) \leq \Delta+6$, equivalently $\varepsilon_\Delta \leq 4$, for all $2\leq\Delta\leq 450$.
  In fact, we even have $\g(\Delta,2) \leq \Delta+4$, for all $170\leq \Delta\leq 450$.
  \item We have $\g(\Delta,2) = \Delta+3$, for every \emph{odd} number $171 \leq \Delta \leq 450$.
  \item The values $\Delta\geq 170$ with $\g(\Delta,2) = \Delta+4$ are even (by the previous observation), and either $\Delta = 2 \bmod 3$ or $\Delta\in\{174, 184, 204, 208, 234, 264\}$.
  \item The even values $\Delta \geq 170$ with $\g(\Delta,2)=\Delta+3$ are given by
  \[
  \Delta\in\{214, 244, 274, 294, 304, 324, 354, 384\}.
  \]
  They all have the form $\Delta=4 \bmod 30$ or $\Delta=24 \bmod 30$.
  The values $\Delta=24 \bmod 30$ are realized by the family in~\eqref{eqn:Delta24mod30}.
  For $\Delta=4 \bmod 30$ we suspect that there is a similar family.
  These families will eventually have size $\leq \Delta+2$.
  \item The smallest even value $\Delta$ with $\Delta+1$ nonprime and $\g(\Delta,2) < \Delta+4$ is $\Delta=132$ with $\g(132,2)=134$.
  \item The only values $\Delta$ which meet the upper bound of \cref{thm:linear-upper-bound} and which do not belong to the families ~\ref{itm:F1}--\ref{itm:F3} are given by
  \[
  \Delta\in\{24, 25, 34, 37, 48, 49, 55, 64, 76, 94, 124, 127, 154, 168, 169, 208\}.
  \] 
  \item The only even values of $\Delta = 2 \bmod 3$ with $\g(\Delta,2) \neq \Delta+4$ are $\g(2,2) = 4 = \Delta+2$ and $\g(62,2) = 67 = \Delta+5$.
  \item The values $\Delta$ with $\g(\Delta,2)=\Delta+5$ are:
  \begin{itemize}
    \item $\Delta\in\{25, 49, 121, 169\}$, i.e.~$\Delta=p^2$ for $p\in\{5, 7, 11, 13\}$
    \item $\Delta\in\{54, 84, 114, 144\}$, i.e.~$\Delta = 24 \bmod 30$, realized by a generic $\Delta$-modular matrix of the form~\eqref{eqn:Delta24mod30}; compare \cref{tab:nuValues}.
    \item $\Delta\in\{37, 55, 62, 127, 142\}$
  \end{itemize}
  \item The values $\Delta$ with $\g(\Delta,2) = \Delta+6$ are $\Delta\in\{24, 48, 120, 168\}$, i.e.~$\Delta=p^2-1$ for $p\in\{5, 7, 11, 13\}$.
\end{itemize}

\begin{figure}[ht]
  \includegraphics*[width=\textwidth]{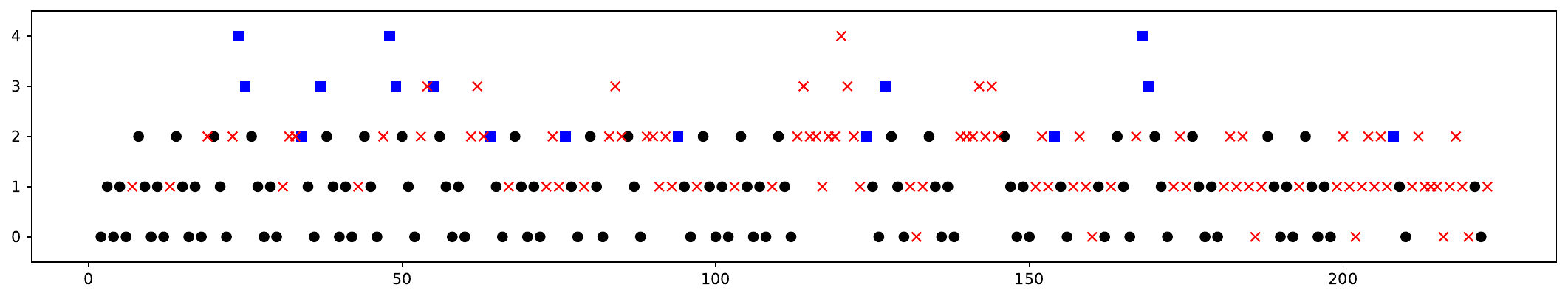} \\[1em]
  \includegraphics*[width=\textwidth]{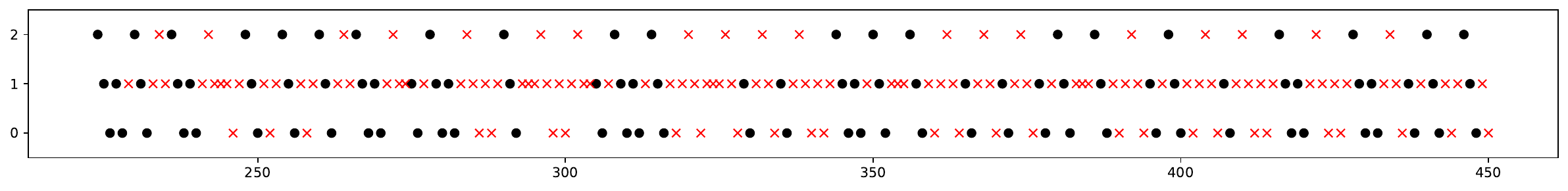}
  \caption{A plot showing $\varepsilon_\Delta = \g(\Delta,2) - (\Delta+2)$, for $2\leq\Delta\leq 450$. The black circle values are those $\Delta$ which belong to one of the families~\ref{itm:F1}--\ref{itm:F3}, the blue square values are other $\Delta$ that meet the upper bound of \cref{thm:linear-upper-bound}, and the red cross values are the remaining $\Delta$ which do not meet the upper bound of \cref{thm:linear-upper-bound}. }
 \label{fig:gDelta2Values}
\end{figure}


From a broad perspective, the data suggests that if~$\Delta$ is large enough, then the behavior of the function $\g(\Delta,2)$ gets much more uniform, and that only for small~$\Delta$ there are exceptional values not following a consistent pattern.

\begin{conjecture}
\label{conj:explicit-form-g-Delta-2}
There is a constant $\delta_2 \in \Z_{>0}$, such that for every $\Delta \geq \delta_2$, we have
  \begin{equation*}
    \g(\Delta, 2) = 
    \begin{cases}
      \Delta + 4 & , \text{if } \Delta \text{ is even and } \Delta = 2 \bmod 3 \\
      \Delta + 3 & , \text{if }\Delta \text{ is odd} \\
      \Delta + 2 & , \text{otherwise}.
    \end{cases}
  \end{equation*}
\end{conjecture}

\noindent Our data suggests that the choice $\delta_2=385$ might work. If this is the case, then this also implies \cref{conj:monotonicity-two-rows}, as the function described in \cref{conj:explicit-form-g-Delta-2} is non-decreasing.

\subsubsection{More than two rows}

For the case of more than two rows, we computationally determined the values of $\g(\Delta,r)$, for every $r \in \{3,4,5\}$ and $2 \leq \Delta \leq 40$.
The results are given in \cref{tab:gDelta345Values} and are depicted in \cref{fig:gDelta345Values}.

Consistent with the sublinear bound in \cref{thm:sublinear-bound-generic-heller} is that the linear upper bound $\g(\Delta,r) \leq p+1$ from \cref{thm:linear-upper-bound} is attained only for certain small parameters~$\Delta$.
In contrast to the case $r=2$, we see that $\Delta \mapsto \g(\Delta,r)$ cannot be a non-decreasing function.
However, again suggesting more uniformity for larger values of~$\Delta$, the following question (which extends \cref{conj:monotonicity-two-rows}) might have an affirmative answer:

\begin{question}
Given $r$, is there a constant $\Delta_r \in \Z_{>0}$, such that $\g(\Delta,r) \leq \g(\Delta+1,r)$, for every $\Delta \geq \Delta_r$ ?
\end{question}

\begin{table}[ht]
  \begin{tabular}{c|*{30}{c}}
    $\Delta$       & 2 & 3 & 4 & 5 & 6 & 7 & 8 & 9 & 10 & 11 & 12 & 13 & 14 &   \\
    $g(\Delta, 3)$ & \textbf{4} & \textbf{6} & \textbf{6} & \textbf{8} & \textbf{8} & 8 & 8 & \textbf{12} & 10 & 12 & 12 & 12 & 14 &   \\
    $g(\Delta, 4)$ & \textbf{5} & \textbf{6} & \textbf{6} & \textbf{8} & 7 & 8 & 9 & 10 & 9 & 10 & 10 & 10 & 11 &   \\
    $g(\Delta, 5)$ & \textbf{6} & \textbf{6} & \textbf{6} & \textbf{8} & \textbf{8} & 8 & 10 & 10 & 10 & 10 & 10 & 10 & 10 &   \\
    \hline
    $\Delta$       & 15 & 16 & 17 & 18 & 19 & 20 & 21 & 22 & 23 & 24 & 25 & 26 & 27 &   \\
    $g(\Delta, 3)$ & 13 & 14 & 14 & 14 & 16 & 16 & 16 & 16 & 17 & 16 & 17 & 18 & 18 &   \\
    $g(\Delta, 4)$ & 10 & 11 & 11 & 12 & 12 & 12 & 12 & 12 & 12 & 12 & 12 & 12 & 12 &   \\
    $g(\Delta, 5)$ & 10 & 10 & 11 & 10 & 11 & 11 & 11 & 11 & 11 & 12 & 11 & 11 & 12 &   \\
    \hline
    $\Delta$       & 28 & 29 & 30 & 31 & 32 & 33 & 34 & 35 & 36 & 37 & 38 & 39 & 40 &   \\
    $g(\Delta, 3)$ & 18 & 18 & 19 & 19 & 20 & 20 & 20 & 20 & 21 & 22 & 22 & 22 & 24 &   \\
    $g(\Delta, 4)$ & 13 & 13 & 13 & 13 & 13 & 13 & 14 & 14 & 14 & 14 & 14 & 14 & 14 &   \\
    $g(\Delta, 5)$ & 11 & 12 & 12 & 12 & 13 & 13 & 12 & 12 & 12 & 12 & 12 & 12 & 13 &   \\
  \end{tabular}
  \caption{The values $\g(\Delta,r)$ for $2 \leq \Delta \leq 40$ and $r=3,4,5$. The values in bold meet the upper bound in \cref{thm:linear-upper-bound}.}
 \label{tab:gDelta345Values}
\end{table}

\begin{figure}[ht]
  \begin{tikzpicture}[scale=0.3, every node/.style={scale=0.7}]
    \tikzstyle{line} = [thick];
    \def\transp{70};
    
    \draw (2,3) -- (2,24);
    \draw (2,3) -- (40,3);
    \node at (20,1.5) {$ \Delta $};
    \node at (-0.5,13.5) {$ \g(\Delta, r) $};

    \draw (2,3) -- (40,3);
    \foreach \x in {2,4,...,40}
    {
      \node at (\x,2.5) {$ \x $};
    }
    \foreach \y in {4,6,...,24}
    {
      \node at (1.3,\y) {$  \y $};
    }
    \foreach \x in {3,...,40}
    {
      \draw[thin, gray] (\x,3) -- (\x,24);
    }
    \foreach \y in {4,...,24}
    {
      \draw[thin,gray] (2,\y) -- (40,\y);
    }
    
    \node at (41.7,24) {$ \g(\Delta, 3) $};
    \draw[black!50,line width=1pt, double]  (2,4) -- (3,6) -- (4,6) -- (5,8) -- (6,8) -- (7,8) -- (8,8) -- (9,12) -- (10,10) -- (11,12) -- (12,12) -- (13,12) -- (14,14) -- (15,13) -- (16,14) -- (17,14) -- (18,14) -- (19,16) -- (20,16) -- (21,16) -- (22,16) -- (23,17) -- (24,16) -- (25,17) -- (26,18) -- (27,18) -- (28,18) -- (29,18) -- (30,19) -- (31,19) -- (32,20) -- (33,20) -- (34,20) -- (35,20) -- (36,21) -- (37,22) -- (38,22) -- (39,22) -- (40,24); 

    \node at (41.7,14) {$ \g(\Delta, 4) $};
    \draw[black!50,line width=2.2pt, solid]  (2,5) -- (3,6) -- (4,6) -- (5,8) -- (6,7) -- (7,8) -- (8,9) -- (9,10) -- (10,9) -- (11,10) -- (12,10) -- (13,10) -- (14,11) -- (15,10) -- (16,11) -- (17,11) -- (18,12) -- (19,12) -- (20,12) -- (21,12) -- (22,12) -- (23,12) -- (24,12) -- (25,12) -- (26,12) -- (27,12) -- (28,13) -- (29,13) -- (30,13) -- (31,13) -- (32,13) -- (33,13) -- (34,14) -- (35,14) -- (36,14) -- (37,14) -- (38,14) -- (39,14) -- (40,14); 

    \node at (41.7,13) {$ \g(\Delta, 5) $};
    \draw[black!50,line width=2.2pt, densely dotted]  (2,6) -- (3,6) -- (4,6) -- (5,8) -- (6,8) -- (7,8) -- (8,10) -- (9,10) -- (10,10) -- (11,10) -- (12,10) -- (13,10) -- (14,10) -- (15,10) -- (16,10) -- (17,11) -- (18,10) -- (19,11) -- (20,11) -- (21,11) -- (22,11) -- (23,11) -- (24,12) -- (25,11) -- (26,11) -- (27,12) -- (28,11) -- (29,12) -- (30,12) -- (31,12) -- (32,13) -- (33,13) -- (34,12) -- (35,12) -- (36,12) -- (37,12) -- (38,12) -- (39,12) -- (40,13); 
    
  \end{tikzpicture}
  \caption{A plot of the values in \cref{tab:gDelta345Values} illustrating the growth of $\g(\Delta,r)$, for $r \in \{3,4,5\}$ and for small $\Delta$.}
  \label{fig:gDelta345Values}
\end{figure}
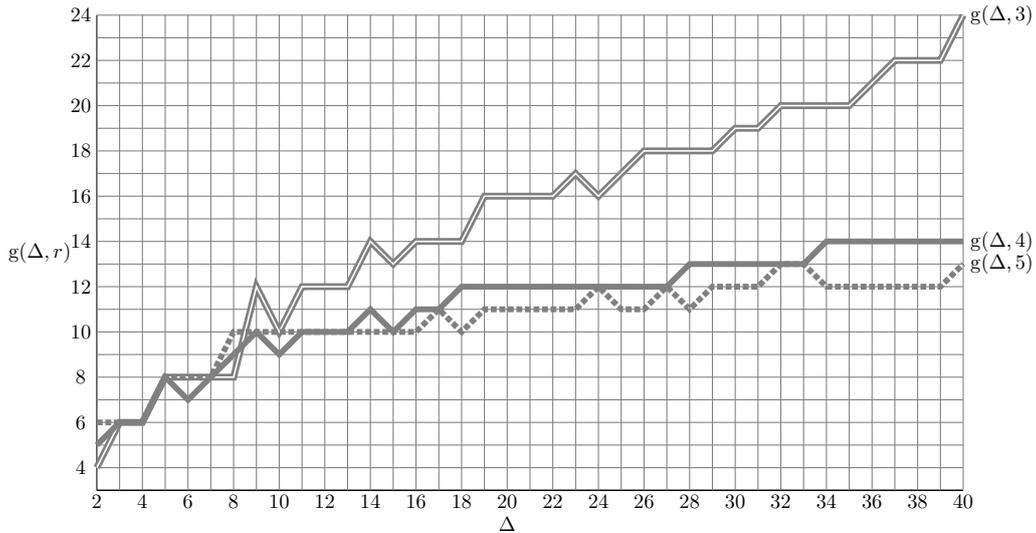

\subsubsection{Non-generic setting}
\label{sect:non-generic-computations}

With only slight modifications the ideas behind \cref{alg:computeHeller} are also applicable in the non-generic setting.
More precisely, we are interested in determining the maximum number of columns that a $\Delta$-modular matrix~$A$ with~$r = \rank(A)$ rows can have.
Compared to the generic case, this means that we allow vanishing minors of size $r \times r$.
The following lemma is the ``non-generic'' version of \cref{lem:existence-Deltabound-totally-gen} and can be proved analogously. 

\begin{lemma}
Let $D \in \Z^{r\times(r+L)}$ be a $\Delta$-modular matrix with $\rank(D)=r$.
Then, there exists a $\Delta$-bound matrix $C\in\Z^{r\times L}$.
Further, each column of $C$ is a solution to a linear system of equations $Ax=0 \bmod \Delta$ for an integer matrix $A\in \Z^{r\times r}$ in Hermite normal form satisfying $\det(A)=\Delta$.
\end{lemma}

While the corresponding counting function $\Delta \mapsto \h(\Delta,r)$ concerns matrices allowing a zero-column and also both~$v$ and~$-v$, for reasons of efficiency, we implemented the corresponding algorithm in a way that both of these are not allowed.
This reduces the size of the hypergraphs roughly by a factor of two and therefore improves the runtime of the algorithm considerably.
However, the hypergraphs still end up very large and also happen to have large cliques.
Therefore, the performance compared to the generic setting is much worse, so that we get more limited data.

The values $\h(\Delta, 3)$, for $3\leq \Delta\leq 11$, have been computed by an alternative algorithm in~\cite{averkovschymura2023onthemaximal}.
With our implementation of \cref{alg:computeHeller} tailored to~$\h(\Delta,r)$, we could compute further values and determined $\h(\Delta,3)$ up to $\Delta \leq 25$.
As a result, we have $\h(\Delta,3)=6\Delta+9 = r^2 + r + 1 + 2r(\Delta-1)$, for all $3\leq \Delta\leq 25, \Delta\neq 4$, which is the size of the construction given in~\cite{leepaatstallknechtxu2022polynomial}.
This gives further evidence that $\Delta=4$ might be the only exception for $r=3$ rows.

For $r=4$ rows, we were able to calculate $\h(\Delta,4)$ in the range $3 \leq \Delta \leq 8$, with the results given in \cref{tbl:hDelta4values}.
For $\Delta \notin \{4,8\}$, the values again are given by the construction in~\cite{leepaatstallknechtxu2022polynomial}, while for $\Delta \in \{4,8\}$, our computations show that the larger examples constructed in~\cite{averkovschymura2023onthemaximal} are best possible.

\begin{table}[ht]
  \begin{tabular}{c|cccccc}
    $\Delta$       &  3 &  4 &  5 &  6 &  7 &  8 \\
    $\h(\Delta,4)$ & 37 & 49 & 53 & 61 & 69 & 81
  \end{tabular}
  \caption{The values $\h(\Delta,4)$ for $3 \leq \Delta \leq 8$. Explicit constructions attaining $\h(4,4)$ and $\h(8,4)$ can be found in \cite{averkovschymura2023onthemaximal}.}
  \label{tbl:hDelta4values}
\end{table}

\noindent An interesting observation is that there are several hypergraphs for which there are a lot of maximum cliques. Therefore it takes a lot of time to go through the search tree of all cliques and the \verb|hClique| algorithm takes considerably longer for these hypergraphs. 
In our implementation of \cref{alg:computeHeller} we parallelized the for loop in Line~\ref{alg:for-loop}. In the generic setting this works well, as most hypergraphs have roughly the same size. However, in the non-generic setting most of the time will be spent waiting on a single thread corresponding to one such large hypergraph. Therefore it might be fruitful to instead parallelize the algorithm \verb|hClique| itself.

\bibliographystyle{amsplain}
\bibliography{generic-delta}

\end{document}